\documentclass[leqno]{article} 
\usepackage{amsfonts}
\usepackage{amsmath, latexsym, amssymb, amsthm, upgreek, amsbsy}
\newcommand{\C}{\mathbb{C}}
\newcommand{\F}{\mathbb{F}}

\newcommand{\PR}{\mathbb{P}}

\newtheorem{theorem}{Theorem}[section] 
\newtheorem{proposition}[theorem]{Proposition} 
\newtheorem{corollary}[theorem]{Corollary} 
\newtheorem{lemma}[theorem]{Lemma} 

\newcommand{\res}{\mathrm{Res}}

\DeclareMathOperator{\Tr}{Tr}
\DeclareMathOperator{\Frob}{Frob}
\DeclareMathOperator{\Prob}{Prob}
\DeclareMathOperator{\PGL}{PGL}
\DeclareMathOperator{\SL}{SL}

\numberwithin{equation}{section}

\begin{document}
\title{Statistics for traces of cyclic trigonal curves over finite fields}

\author{Alina Bucur\footnote{Institute for Advanced Study,
\texttt{alina@math.ias.edu}}, Chantal David\footnote{Concordia University,
\texttt{cdavid@mathstat.concordia.ca}}, Brooke Feigon\footnote{University of
Toronto, \texttt{bfeigon@math.toronto.edu}} and Matilde
Lal\'in\footnote{University of Alberta, \texttt{mlalin@math.ualberta.ca}}}



\date\today 



\maketitle

\begin{abstract}

We study the variation of the trace of the Frobenius endomorphism associated to a cyclic trigonal curve of genus $g$ over $\mathbb F_q$ as the curve varies in an irreducible component of the moduli space. We show that for $q$ fixed and $g$ increasing, the limiting distribution of the trace of Frobenius equals the sum of $q+1$ independent random variables taking the value $0$ with probability $2/(q+2)$ and $1, e^{2\pi i/3} , e^{4 \pi i /3}$ each with probability $q/(3(q+2)).$ This extends the work of Kurlberg and Rudnick who considered the same limit for hyperelliptic curves. We also show that when both $g$ and $q$ go to infinity, the normalized trace has a standard complex Gaussian distribution and how to generalize these results to $p$-fold covers of the projective line. 
\end{abstract}


\begin{center}
\emph{To epsilon}
\end{center}

\setcounter{tocdepth}{1}
\tableofcontents
\section{Introduction}
\label{introduction}

Let $\F_q$ be the finite field with $q$ elements.
For any smooth projective curve $C$ of genus $g$ over $\F_q$,
let $Z_{C}(T)$ be its zeta function. It was shown by Weil \cite{We} that
\[
Z_{C}(T) = \frac{P_C(T)}{(1-T)(1-qT)},
\]
with $$P_C(T) = \prod_{j=1}^{2g} (1 - \alpha_{j}(C) T),$$
and
\[\vert \alpha_j(C) \vert = q^{1/2}, \; \mbox{for $1 \leq j \leq 2g.$}\]
The trace of
the Frobenius endomorphism (acting on the first cohomology group $H^1$) is
then
\[
\Tr(\Frob_C) = \sum_{j=1}^{2g} \alpha_j(C).
\]

We study in this paper the variation of the
trace of the Frobenius endomorphism $\Frob_C$ over moduli spaces of
cyclic trigonal curves of genus $g$ when $g$ tends to infinity.
This
extends the work of Kurlberg and Rudnick who considered the same
limit for the case of hyperelliptic curves \cite{KR}. All of the results extend further
to the case of cyclic $p$-fold covers of $\mathbb{P}^1$ for $p$ prime, as we indicate briefly
in Section~\ref{gen}. However, we have chosen to focus on the trigonal case
because it exhibits all of the essential features of the general case but with
a somewhat lighter notational load. (Note that some of these features do not
appear in \cite{KR}, including reducibility of the moduli space, complex-valued random variables,
and use of the Tauberian theorem.)

Before describing our main results, we describe a modified version of the main
theorem of \cite{KR}.
In the work of Kurlberg and
Rudnick, the statistics are computed for the family of hyperelliptic
curves $Y^2=F(X)$ by running over all square-free polynomials $F$ of
a fixed degree $d$. This is not the same as running over the moduli
space of hyperelliptic curves of a fixed genus, as not all points on
the moduli space appear with the same multiplicity in this family.
Also, the results of \cite{KR} are about the affine trace of
the hyperelliptic curves, which differ slightly from $\Tr(\Frob_C)$.
The geometric version of the work of Kurlberg and Rudnick is then
the following theorem, which is proved in Section \ref{KRrevisited}.

\begin{theorem} \label{KRrevisitedthm}
If $q$ is fixed and $g\rightarrow \infty$, the distribution of the
trace of the Frobenius endomorphism associated to $C$ as $C$ ranges
over the moduli space $\mathcal H_g$ of hyperelliptic curves of
genus $g$ defined over $\mathbb F_q$ is that of a sum of $q+1$
i.i.d. random variables $X_1, \ldots, X_{q+1}$ that take
the value $0$ with probability $1/(q+1)$ and $\pm 1$ each with
probability $1/(2(1+q^{-1})).$ More precisely,  for any $s \in \mathbb
Z$ with $|s|\leq q+1$, we have
 \[\frac{\left|\left\{C \in \mathcal H_g:  \Tr(\Frob_C) = -s\right\}\right|'}{\left|\mathcal H_g\right|'}  = \Prob
 \left(\sum_{i=1}^{q+1} X_i = s \right) \left( 1+ O\left( q^{(3q -2-2g)/2}   \right)\right). \]
\end{theorem}

In the last theorem, and in the rest of the paper, the $'$ notation, applied both to summation and cardinality,
means that curves $C$ on the moduli spaces are  counted with the
usual weights $1/|\mbox{Aut}(C)|$.

For the rest of the paper, we assume that $q \equiv 1
(\mbox{mod } 3)$. For any cube-free polynomial $F \in \F_q[X]$, let $C_F$ be
the cyclic trigonal curve \begin{eqnarray} \label{curveCF} C_F : Y^3
= F(X).\end{eqnarray}
For cyclic trigonal curves, the genus is not a function of the
degree of the polynomial $F$ in (\ref{curveCF}), as it is for the
hyperelliptic curves $Y^2 = F(X)$. Also, the moduli space
 of cyclic trigonal curves of genus $g$ is not
irreducible, and we look at the distribution of $\Tr(\Frob_C)$ on
each irreducible component (Theorem \ref{componentd1d2}). It turns
out to be independent of the component when certain conditions are met.

Let $\mathcal H_{g,3}$ denote the moduli space of cyclic trigonal
curves of genus $g$.  Its irreducible components are determined by a finer geometric invariant, namely the signature. For $d_1+2d_2 \equiv 0 \pmod 3$, let $\mathcal
H^{(d_1, d_2)}$ be the component of the moduli space with curves of
signature $(r,s)= \left( (2d_1+d_2-3)/3,
(d_1+2d_2-3)/3 \right)$. Then
\[
\mathcal H_{g,3}= \bigcup_{\substack{d_{1}+2d_{2} \equiv 0 \pmod
3,\\ g=d_1+d_{2}-2}} \mathcal H^{(d_{1},d_{2})},
\]
where the union is disjoint and each component $\mathcal H^{(d_{1},d_{2})}$ is
irreducible.

 Fix a cubic character $\chi_3$ of $\mathbb F_q$ (recall that $q
\equiv 1 \,(\mathrm{mod}\, 3)$). It takes values $0, 1, \omega$ and
$\omega^2$, where $\omega$ is a primitive third root of unity in
$\mathbb C.$ Each cyclic trigonal curve $C$ is endowed with a cyclic
order 3 automorphism that splits the first cohomology group of $C$
into two subspaces,  $H^1_{\chi_3}$ and $H^1_{\overline \chi_3}$, on
which the automorphism acts via  $\chi_3$ or via its conjugate. Since this
automorphism commutes with the action of the Frobenius, it follows
that
\[\Tr(\Frob_C |_{H^1_{\chi_3}} )= \overline{\Tr(\Frob_C |_{H^1_{\overline
\chi_3}})} . \]
So it is enough to study the distribution of the trace
of the Frobenius on one of these two subspaces.

\begin{theorem}
\label{componentd1d2} If $q$ is fixed and $d_1, d_2 \rightarrow
\infty$, the distribution of the trace of the Frobenius endomorphism
associated to $C$ as $C$ ranges over the component $\mathcal
H^{(d_1, d_2)}$ of cyclic trigonal curves defined over $\mathbb F_q$
is that of the sum of $q+1$ i.i.d.\ random
variables $X_1, \ldots, X_{q+1}$,  where each $X_i$ takes
the value $0$ with probability $2/(q+2)$ and $1, \omega, \omega^2$ each
with probability $q/(3(q+2)).$ More precisely, for any $s \in \mathbb
Z[\omega] \subset \mathbb C$ with $|s|\leq q+1$, we have for any $1>\varepsilon >0$,
\begin{eqnarray*}
\frac{\left|\left\{C \in \mathcal H^{(d_1, d_2)}:\Tr
(\Frob_C|_{H^1_{\chi_3}})=-s \right\}\right|'}{\left|\mathcal H^{(d_1,d_2)}\right|'} = \Prob \left( \sum_{i=1}^{q+1} X_i =  s \right)\left( 1 + O
\left(q^{-(1-\varepsilon)d_2+ q}+q^{-(d_1-3q)/2}\right)
\right).
\end{eqnarray*}
\end{theorem}

It may not be clear where the  probabilities attached to the random variables come from, but they are quite natural, as the heuristic in Section \ref{htri} shows.

We remark that in Theorem \ref{componentd1d2}, and in all the results in our paper,
the implied constants in the error terms are independent of $q$, even when $q$ is fixed. Then, as was done in \cite{KR} for hyperelliptic curves,
we
can also study the case where $q$ and
$d_1, d_2$ tend to infinity. Since the trace takes complex values, the limiting distribution will be the  \emph{complex} Gaussian with mean $0$ and variance $1$, instead of the usual real-valued Gaussian one gets for hyperelliptic curves.
We first compute the moments of $\Tr (\Frob_C|_{H^1_{\chi_3}})/\sqrt{q+1}$ and
compare them with the moments of the normalized sum of the i.i.d.~random variables of Theorem \ref{componentd1d2}.

\begin{theorem}\label{momentsthm}
For any positive integers $j$ and $k$, let $M_{j,k}(q, (d_1, d_2))$ be the
moments
\begin{eqnarray*}
M_{j,k}(q, (d_1, d_2)) &=& \frac{1}{\left|\mathcal{H}^{(d_1, d_2)}\right|^{\;'}}   \sideset{}{'}{\sum}_{C \in \mathcal{H}^{(d_1, d_2)}}
\left( \frac{-\Tr
(\Frob_C|_{H^1_{\chi_3}})}{\sqrt{q+1}} \right)^j \left( \frac{-\Tr
(\Frob_C|_{H^1_{\overline\chi_3}})}{\sqrt{q+1}} \right)^k.
\end{eqnarray*}
Let $\varepsilon$ and  $X_1, \dots, X_{q+1}$ be as in Theorem \ref{componentd1d2}. Then
\begin{eqnarray*}
M_{j,k}(q,(d_1,d_2)) = \mathbb E
\left( \left( \frac{1}{\sqrt{q+1}} \sum_{i=1}^{q+1} X_i \right)^j\left( \frac{1}{\sqrt{q+1}} \sum_{i=1}^{q+1} \overline{X_i} \right)^k
\right)
\left( 1 + O \left(
q^{-(1-\varepsilon)d_2 + \varepsilon (j+k) }+
q^{-d_1/2+j+k} \right)\right).
\end{eqnarray*}
\end{theorem}
\begin{corollary} \label{momentscor} When $q, d_1, d_2$  tend to infinity, the limiting distribution of the normalized trace
\\$\Tr (\Frob_C|_{H^1_{\chi_3}})/\sqrt{q+1}$ is a complex Gaussian with mean zero and variance one.
\end{corollary}

We remark that when $g$ is fixed and $q$ tends to infinity, $\Tr (\Frob_C|_{H^1_{\chi_3}})$ should be distributed as the trace of matrices in a group of random matrices determined by  the monodromy group  of the moduli space of $C$ in the philosophy of Katz and Sarnak \cite{KS}. The monodromy groups for cyclic trigonal curves are computed in \cite[Theorem 3.8]{AP}. Roughly speaking, the  monodromy of each component $\mathcal H^{(d_1,d_2)}$ of signature $(r,s)$ of the moduli space $\mathcal H_{g, 3}$ is an extension of  the group of sixth roots of unity $\pmb{\upmu}_6$ by the special unitary group $\textrm{SU}(r,s)$. The monodromy for the component $\mathcal H^{(d,0)}$ is computed in \cite[Theorem 5.4]{Ka}, and the result is an extension of  $\pmb{\upmu}_6$ by $\SL(g).$

 The structure of this paper is as follows. In Section~\ref{modulispaces}, we describe moduli spaces of
cyclic trigonal curves, and present some notations and results which will be used in the
rest of the paper. In Section~\ref{geometry}, we describe how
$\Tr (\Frob_C|_{H^1_{\chi_3}})$ can be written as a sum of $q+1$ values of
the cubic character of $\chi_3$ of $\F_q$, and how to compute statistics of the
trace by counting classes of the moduli spaces. The proof of Theorem~\ref{componentd1d2} is concluded in Section~\ref{generalcount}.  We compute the moments of $\Tr (\Frob_C|_{H^1_{\chi_3}})/\sqrt{q+1}$ and prove Corollary~\ref{momentscor}
in Section~\ref{sectionmoments}, and we revisit the
case of hyperelliptic curves in Section~\ref{KRrevisited}. In Section~\ref{gen} we explain how the techniques employed in the study of cyclic trigonal case can be adapted to the general case of cyclic $p$-fold covers of $\mathbb P^1(\mathbb F_q)$.  Finally, we
present in Section~\ref{heuristic} a heuristic model which predicts
the results obtained in Theorems~\ref{KRrevisitedthm},
\ref{componentd1d2} and \ref{comp-gen}.

\section{Setting and notation}
\label{modulispaces}

Fix  $q\equiv 1 \pmod  3$. We will denote by $\zeta_q$  the
(incomplete) zeta function of the rational function field $\mathbb
F_q[X]$ given by
\begin{eqnarray*}
\zeta_q(s) = \sum_{F} |F|^{-s} = \prod_{P} \left( 1 - |P|^{-s}
\right)^{-1} = (1 - q^{1-s})^{-1}.
\end{eqnarray*}

Let $C$ be a cyclic trigonal curve over $\F_q$, i.e. a cyclic cover
of order 3 of $\PR^1$ defined over $\F_q$. Then, $C$ has an affine
model $Y^3=F(X)$, where $F(X)$ is a polynomial in $\F_q[X]$. If
$G(X) = H(X)^3 F(X)$, then $Y^3=F(X)$ and $Y^3=G(X)$ are isomorphic over
$\F_q$, so it suffices to consider curves $Y^3=F(X)$ with $F(X)$
cube-free.

Let $F \in \F_q[X]$ be cube-free and monic. Recall that cube-free over $\mathbb F_q$ is the same as cube-free over $\overline{\F}_q$. So $F$ factors in
$\overline {\mathbb F}_q[X]$ as
\[F(X) = \prod_{i=1}^{d_1} (X-a_i)\prod_{j=1}^{d_2}(X-b_j)^2,\]
where $a_i, b_j$ are distinct elements of $\overline{\F}_q$.

Let $C_F$ be the cyclic trigonal curve given by $Y^3 = F(X) = F_1(X)
F_2(X)^2,$ with $F_1$ and $F_2$ relatively prime, square-free, $\deg{F_1}=d_1$, $\deg{F_2}=d_2$ and $d = \deg{F} =
d_1 + 2d_2$.  The curve $C_F$ has genus $g$ if and
only if $d_1 +2d_2 \equiv 0 \pmod 3$ and $g=d_1+d_2-2$, or $d_1
+2d_2 \equiv 1 \text{ or } 2 \pmod 3$ and $g=d_1+d_2-1$. Over $\overline {\mathbb F}_q$, one can reparametrize
and choose an affine
model for any cyclic trigonal curve with $d_1+2d_2 \equiv 0 \,(\mathrm{mod}\, 3)$. We already see that the
relationship between the genus of the curve $C_F$ and the degree of
the polynomial $F$ defining it is not as simple as in the
hyperelliptic case, as the genus is not a function of the degree.
Note that by interchanging $F_1$ and $F_2$, we are replacing $F$ by
$F^2$ (modulo a perfect cube) and the two curves $Y^3 = F_1(X)
F_2(X)^2$ and $Y^3 = F_1(X)^2 F_2(X)$ are isomorphic.  Furthermore,
the moduli space $\mathcal H_{g,3}$ of cyclic trigonal curves of
fixed genus $g$ splits into irreducible subspaces indexed by pairs
of nonnegative integers $d_1, d_2$ with the property that $d_1 +
2d_2 \equiv 0 \,(\mathrm{mod}\, 3)$, and the moduli space can be written as a disjoint union over its connected components
\begin{eqnarray} \label{modulicomponents}
\mathcal H_{g,3}= \bigcup_{\substack{d_{1}+2d_{2} \equiv 0 \pmod
3,\\ g=d_1+d_{2}-2}} \mathcal H^{(d_{1},d_{2})}.
\end{eqnarray}
Each component $\mathcal H^{(d_{1},d_{2})}$ is irreducible, and pairs $(d_1,d_2)$ and $(d_2,d_1)$ give the same component.

The components can also be described by their signature $(r,s)$. The
signature and $(d_1, d_2)$ are related by $d_1 = 2r-s+1$ and
$d_2=2s-r+1$, or equivalently $r=(2d_1+d_2-3)/3$ and $s=
(d_1+2d_2-3)/3$. Each unordered pair $\{r,s\}$ represents a
different component of the moduli space of cyclic trigonal curves.
We refer the reader to \cite{AP} for the details.

In view of the previous observations, we will write
\[F(X)=F_1(X)F_2(X)^2,\]
where $F_1$ and $F_2$ are relatively prime monic square-free
polynomials with $\deg F_1=d_1$ and  $\deg F_2=d_2$.

We will use the following sets of polynomials:
\begin{eqnarray*}
V_d &=& \left\{ F \in \mathbb F_q[X] \;:\; F \textrm{ monic,} \deg F = d\right\}\\
\mathcal F_d & = & \{F\in \mathbb F_q[X]:  F \textrm{ monic,  square-free and}\deg F=d   \}  \\
\widehat{\mathcal F}_d & = & \{F\in \mathbb F_q[X]:  F \textrm{ square-free and}\deg F=d   \}  \\
\mathcal F_{(d_1, d_2) }&=&
\{F= F_1 F_2^2 : F_1, F_ 2 \textrm{ monic, square-free and coprime},  \deg F_1 = d_1, \deg F_2 = d_2  \}  \\
\mathcal F_{(d_1, d_2)}^k &=&\left\{ F=F_1 F_2^2 \in \mathcal F_{(d_1, d_2)}: F_2 \text{ has } k \text{ roots in } \mathbb F_q\right\} \\
\widehat{\mathcal F}_{(d_1, d_2) }&=& \left\{F= \alpha F_1 F_2^2 \;:\;
\alpha \in
\F_q^*, F_1F_ 2^2 \in \mathcal F_{(d_1,d_2)}  \right\}
\\
{\mathcal{F}}_{[d_1, d_2]} &=& {\mathcal{F}}_{(d_1, d_2)} \cup
{\mathcal{F}}_{(d_1-1, d_2)} \cup {\mathcal{F}}_{(d_1, d_2-1)}
\\
\widehat{\mathcal{F}}_{[d_1, d_2]} &=& \widehat{\mathcal{F}}_{(d_1,
d_2)} \cup \widehat{\mathcal{F}}_{(d_1-1, d_2)} \cup
\widehat{\mathcal{F}}_{(d_1, d_2-1)} .
\end{eqnarray*}

As a matter of convention, from now on, all our  polynomials will be
monic unless otherwise stated. Also, we will use $P$ to denote monic irreducible polynomials.

We transcribe here the relevant results from the work of Kurlberg and Rudnick \cite{KR}.

\begin{lemma}\cite[Lemma 3]{KR} \label{KRlemma3}
The number of square-free monic polynomials of degree $d$ is
\[|{\mathcal F}_d|=\left \{ \begin{array}{ll} {q^{d}}{(1-q^{-1})} & d\geq 2,\\\\q^{d} & d=0,1.\end{array}\right.\]
\end{lemma}

\begin{lemma}\cite[Lemma 4]{KR} \label{KRlemma4}For $0 \leq \ell\leq q$, let $x_1, \ldots, x_\ell
\in \mathbb F_q$  be distinct elements, and let $a_1, \ldots, a_\ell
\in \mathbb F_q$. If $d \geq \ell$, then
\[\left|\left\{F \in V_d: F(x_1) = a_1, \ldots, F(x_\ell)=a_\ell \right\}\right|= q^{d-\ell}. \]
\end{lemma}

\begin{lemma}\cite[Lemma 5]{KR}  \label{KRlemma5}  Let $d\geq 2$ and $\ell
\leq q$ be positive integers, let  $x_1, \ldots, x_\ell \in \mathbb
F_q$  be distinct elements, and let $a_1, \ldots, a_\ell \in \mathbb
F_q$ be nonzero elements. Then
\[\left|\left\{F \in \mathcal F_d: F(x_1) = a_1, \ldots, F(x_\ell)=a_\ell \right\}\right|=  \frac{q^{d-\ell}}{\zeta_q(2) (1 - q^{-2})^\ell}
+ O\left( q^{d/2}\right).\] \end{lemma}

\begin{lemma}\cite[Proposition 6]{KR}  \label{KRproposition6}
Let $ x_1, \ldots, x_{\ell+m} \in \mathbb F_q$ be distinct elements,
let $a_1, \ldots, a_\ell \in \mathbb F^*_q$, and let $a_{\ell+1}= \cdots=
a_{\ell+m}=0$. Then
\[ \left|\left\{F \in \mathcal F_d: F(x_i)= a_i, 1 \leq i \leq m+\ell\right\}\right|
=  \frac{(1-q^{-1})^m q^{d-(m+\ell)}}{\zeta_q(2) (1-q^{-2})^{m+\ell}}
\left( 1+O\left( q^{(3m+2\ell - d)/2}\right) \right). \]
and
\[ \frac{\left|\left\{F \in \mathcal F_d: F(x_i)= a_i, 1 \leq i \leq m+\ell\right\}\right|}{\left| \mathcal{F}_d\right|}
=  \frac{(1-q^{-1})^m q^{-(m+\ell)}}{ (1-q^{-2})^{m+\ell}} \left( 1+O\left( q^{(3m+2\ell - d)/2}\right) \right). \]
\end{lemma}

We will also use the following Lemma which follows easily from Lemma
\ref{KRproposition6}.

\begin{lemma} \label{fromKR} Fix $0 \leq k \leq q$. Then
\[\left|\mathcal{F}_{d}^k\right|=\frac{\binom{q}{k}q^{d-k}}{\zeta_q(2)
(1+q^{-1})^q} \left(1+O\left(q^{(k+2q-d)/2}\right) \right).\]
\end{lemma}
\begin{proof}
In the Lemma \ref{KRproposition6}, set $m=k$ and $\ell=q-k$. In this way, we guarantee that there are exactly $k$ zeros. Now we also have $\binom{q}{k}$ options for choosing the zeros, and $(q-1)^\ell$ options for choosing the nonzero values. Combining all of this with Lemma \ref{KRproposition6}, we get the formula.
\end{proof}

\section{The geometric point of view}
\label{geometry}

We prove in this section that Theorem \ref{componentd1d2} follows
from the following theorem which will be proved in Section
\ref{generalcount}. Recall that  $\chi_3$ is a fixed cubic character
of $\mathbb F_q$ and $\omega$ is a primitive third root of unity in $\C$.

\begin{theorem} \label{countingF} Let $x_1, \dots, x_q$ be the elements of $\F_q$ and
let $\varepsilon_1, \dots, \varepsilon_q \in \left\{ 0,1,\omega,
\omega^2 \right\}$. Let $m$ be the number of values of $\varepsilon_i$
which are 0. Then for any $\varepsilon > 0$
\begin{eqnarray*}
\left| \mathcal{F}_{(d_1, d_2)} \right| &=& \frac{K q^{d_1+d_2}}{\zeta_q(2)^2}
\left(1+ O\left(q^{-(1-\varepsilon )d_2} + q^{-d_1/2}\right) \right),
\end{eqnarray*}
\begin{eqnarray}\nonumber
\left|\left\{ F \in \mathcal{F}_{(d_1,d_2)} \;:\; \chi_3(F(x_i)) =
\varepsilon_i, \; 1 \leq i \leq q \right\}\right| &=&\frac{K q^{d_1+d_2}}{\zeta_q(2)^2}
\left( \frac{2}{q+2} \right)^m \left( \frac{q}{3(q+2)}
\right)^{q-m}
\\ \label{useful}
&\times&\left(1+ O\left( q^{ -(1-\varepsilon) (d_2-m)+\varepsilon q}+q^{-(d_1-m)/2+q} \right)\right),
\end{eqnarray}
and
\begin{eqnarray*}
 \frac{\left|\left\{ F \in \mathcal{F}_{(d_1,d_2)} \;:\;
\chi_3(F(x_i)) = \varepsilon_i, \; 1 \leq i \leq q
\right\}\right|}{\left|\mathcal{F}_{(d_1,d_2)} \right|}
&=&\left( \frac{2}{q+2} \right)^m \left( \frac{q}{3(q+2)}
\right)^{q-m}\\&\times&\left(1+ O\left( q^{ -(1-\varepsilon) (d_2-m)+\varepsilon q}+q^{ -(d_1-m)/2+q} \right)\right),
\end{eqnarray*}
where $K$ is the constant
\[ K =  \prod_{P} \left( 1 - \frac{1}{(|P|+ 1)^2}
\right).\]
\end{theorem}

For any
polynomial $F \in \mathcal{F}_{(d_1,d_2)}$, let
\begin{eqnarray} \label{definitionSF}
S_3(F) = \sum_{x \in \F_q} \chi_3(F(x)).
\end{eqnarray}
Then, the number of affine points on the curve $Y^3=F(X)$ is given by
\[
\sum_{x \in \F_q} 1 + \chi_3(F(x)) + \overline{\chi_3(F(x))} = q +
S_3(F) + \overline{S_3(F)}.
\]

Using  Theorem \ref{countingF}, we can immediately deduce a result
for the distribution of the affine trace $- (S_3(F) +
\overline{S_3(F)})$ when we vary over the family of curves $C_F: Y^3
= F(X)$ for $F(X) \in \mathcal{F}_{(d_1,d_2)}$. This is the
``non-geometric version" of Theorem \ref{componentd1d2} which
corresponds to Theorem 1 of \cite{KR}. When comparing Theorem
\ref{componentd1d2} and Corollary \ref{non-geo-thm}, it is
interesting to remark that the point at infinity appearing in the
trace of the Frobenius on $H^1_{\chi_3}$, and not in the affine
trace, behaves like any other point.

\begin{corollary} \label{non-geo-thm}Let $X_1, \ldots, X_q$ be $q$ i.i.d. random
variables taking the value $0$ with probability $2/(q+2)$ and any of
the values $1, \omega, \omega^2$ with probability $q/(3(q+2))$. Then for any $\varepsilon > 0$,
\begin{eqnarray*}
\frac{\left|\left\{ F \in \mathcal{F}_{(d_1,d_2)} \;:\; S_3(F) = s
\right\}\right|}{\left|\mathcal{F}_{(d_1,d_2)} \right|} &=& \Prob
\left( \sum_{i=1}^q X_i = s \right)\left(1+ O\left( q^{ -(1-\varepsilon) d_2+q}+q^{-(d_1-3q)/2} \right)\right)
\end{eqnarray*}
for any $s \in \mathbb Z[\omega]\subset \mathbb C$.
\end{corollary}
\begin{proof} Using (\ref{definitionSF}), we write
\begin{eqnarray*} && \hspace{-0.4in} \frac{\left|\left\{ F \in \mathcal{F}_{(d_1,d_2)} \;:\; S_3(F) =  s
\right\}\right|}{\left|\mathcal{F}_{(d_1,d_2)} \right|} \\&=&
\sum_{{(\varepsilon_1, \dots, \varepsilon_q) \in \left\{ 0, 1, \omega,
\omega^2 \right\}} \atop {\varepsilon_1 + \dots + \varepsilon_q = s}}
\frac{\left|\left\{ F \in \mathcal{F}_{(d_1,d_2)} \;:\;
\chi_3(F(x_i)) = \varepsilon_i, \; 1 \leq i \leq q
\right\}\right|}{\left|\mathcal{F}_{(d_1,d_2)} \right|}\\
&=& \sum_{{(\varepsilon_1, \dots, \varepsilon_q) \in \left\{ 0, 1,
\omega, \omega^2 \right\}} \atop {\varepsilon_1 + \dots + \varepsilon_q
= s}}  \left( \frac{2}{q+2} \right)^m \left( \frac{q}{3(q+2)}
\right)^{q-m}\left(1+ O\left( q^{ -(1-\varepsilon) (d_2-m)+\varepsilon q}+q^{-(d_1-m)/2+q} \right)\right)\\
&=& \Prob \left( \sum_{i=1}^q X_i = s \right)\left(1+ O\left( q^{ -(1-\varepsilon) (d_2-q)+\varepsilon q}+q^{-(d_1-q)/2+q} \right)\right).
\end{eqnarray*}
\end{proof}

\begin{proof}[Proof of Theorem \ref{componentd1d2}.]
We now proceed to the proof of Theorem \ref{componentd1d2}, assuming
Theorem \ref{countingF}.  When we write a cyclic trigonal curve as
\begin{equation} \label{trigonal}
C_F: Y^3=F(X)
\end{equation}
where $F(X)$ is cube-free, we are choosing an affine model of the
curve. To compute the statistics for the components $\mathcal
H^{(d_1, d_2)}$ of the moduli space $\mathcal H_{g, 3}$, we need to
work with families where we count each curve, seen as a projective
variety of dimension $1$, up to isomorphism, with the same
multiplicity. To do so, we have to consider all cube-free
polynomials in $\F_q[X]$, and not only monic ones. We fix a
genus $g$, and a component $\mathcal{H}^{(d_1, d_2)}$ for this genus
as in equation (\ref{modulicomponents}). For each point of this component, we want to count its affine models   $C': Y^3 =
G(X)$ with $G \in \widehat{\mathcal{F}}_{[d_1, d_2]}$.

For $g \geq 5$, the curves $C'$ isomorphic to $C$ are obtained from
the automorphisms of $\PR^1(\F_q)$, namely the $q(q^2-1)$ elements
of $\PGL_2(\F_q)$. By running over the elements of $\PGL_2(\F_q)$,
we obtain $q(q^2-1)/|\mbox{Aut}(C)|$ different models $C': Y^3 = G(X)$
where  $G \in \widehat{\mathcal{F}}_{[d_1, d_2]}$. This shows that
\begin{eqnarray} \label{correctcount}
|\mathcal H^{(d_{1},d_{2})}|'=\sideset{}{'}{\sum}_{C \in \mathcal H^{(d_{1},d_{2})}} 1 = \sum_{C \in \mathcal H^{(d_{1},d_{2})}} \frac{1}{|\mbox{Aut}(C)|} &=&
\frac{ |\widehat{\mathcal F}_{[d_{1},d_{2}]}|}{q(q^{2}-1)}.
\end{eqnarray}

We denote
\begin{eqnarray*}
\widehat{S}_3(F) &=& \sum_{x \in \mathbb P^1(\mathbb F_q)}
\chi_3(F(x)),\end{eqnarray*} where the value of $F$ at the point at
infinity is defined below. 
Fix  an enumeration of the points on $\PR^1(\mathbb F_q)$, $x_1,
\dots, x_{q+1}$, such that $x_{q+1}$ denotes the point at infinity.
Then
\[
F(x_{q+1})= \begin{cases} \text{leading coefficient of $F$} & F \in
\widehat{\mathcal F}_{(d_{1},d_{2})}, \\ 0 & F \in \widehat{\mathcal
F}_{(d_{1}-1,d_{2})} \cup \widehat{\mathcal F}_{(d_{1},d_{2}-1)}.
\end{cases}
\]
Therefore, $\widehat{S}_3(F)+\overline{\widehat{S}_3(F)}$ is equal
to
\begin{align*}
S_3(F)+\overline{S_3(F)}+
 \begin{cases}
2 & F \in \widehat{\mathcal F}_{(d_1, d_2)} \text{ and leading
coefficient of $F$ is a cube,}
\\
-1 & F \in \widehat{\mathcal F}_{(d_1, d_2)} \text{ and leading
coefficient of $F$ is not a cube,}
\\
0 & F \in \widehat{\mathcal F}_{(d_1-1, d_2)}  \cup
\widehat{\mathcal F}_{(d_{1},d_{2}-1)}.
\end{cases}
\end{align*}
Then, the number of points on the projective curve $C_F$ with affine
model (\ref{trigonal}) is given by
\[
\sum_{x \in \mathbb P^1(\mathbb F_q)} 1 + \chi_3(F(x)) +
\overline{\chi_3(F(x))} = q+1+\widehat{S}_3(F) +
\overline{\widehat{S}_3(F)}
\]
and
\begin{eqnarray} \label{valueoffrob}
\Tr (\Frob_C|_{H^1_{\chi_3}}) &=& - \widehat{S}_3(F)\\
\Tr (\Frob_C|_{H^1_{\overline \chi_3}}) &=& - \overline {\widehat S _3(F)}.
\end{eqnarray}
As in (\ref{correctcount}), we write
\begin{eqnarray} \label{otherone}
\left|\left\{C \in \mathcal H^{(d_1,d_2)}:\Tr (\Frob_C|_{H^1_{\chi_3}})= - s \right\}\right|' = \sum_{{C
\in \mathcal{H}^{(d_1,d_2)}} \atop {\Tr(\Frob_C|_{H^1_{\chi_3}})=-s}} \frac{1}{|\mbox{Aut}(C)|}.
\end{eqnarray}
It then follows from (\ref{correctcount}), (\ref{valueoffrob}) and
(\ref{otherone}) that
\begin{eqnarray} \label{countH}
\frac{\left|\left\{C \in \mathcal H^{(d_1,d_2)}:\Tr (\Frob_C|_{H^1_{\chi_3}})=-s
\right\}\right|'}{\left|\mathcal H^{(d_1,d_2)}\right|'}   &=& \frac{\left|\left\{F \in
\widehat{\mathcal F}_{[d_{1},d_{2}]} \;:\; \widehat{S}_3(F)=s\right\}\right|}{
\left|\widehat{\mathcal F}_{[d_{1},d_{2}]} \right|}.
\end{eqnarray}

We now rewrite (\ref{countH}) in terms of polynomials in ${\mathcal
F}_{(d_1, d_2)}$. We first compute
\begin{eqnarray} \nonumber |\widehat{\mathcal{F}}_{[d_1,d_2]}|
  &=&
(q-1)\left(\left|\mathcal{F}_{(d_1,d_2)}\right|+
\left|\mathcal{F}_{(d_1-1,d_2)}\right|+
\left|\mathcal{F}_{(d_1,d_2-1)}\right|\right ) \\ \label{numberpoly}&= & \frac{K}{\zeta_q(2)^2}
\frac{(q+2)(q-1)}{q} q^{d_1+d_2} \left(1+ O\left(q^{-(1-\varepsilon) d_2} +q^{-d_1/2 }  \right)\right)
\end{eqnarray}
by Theorem \ref{countingF}.

Fix a $(q+1)$-tuple $(\varepsilon_1, \dots,
\varepsilon_{q+1})$ where $\varepsilon_i  \in \{ 0, 1, \omega,
\omega^2\}$ for $1 \leq i \leq q+1$. Denote by $m$ the number of $i$
such that $\varepsilon_i=0$. We want to evaluate the probability
that the character $\chi_3$ takes exactly these values at the points
$F(x_1), \dots, F(x_{q+1})$ where  $x_{q+1}$ is the point at
infinity of $\PR^1(\F_q)$, as $F$ ranges over $ \widehat{\mathcal
F}_{[d_{1},d_{2}]}$.

\textbf{Case 1:} $\varepsilon_{q+1}=0$.

In this case, only polynomials from $\widehat{\mathcal
F}_{(d_{1}-1,d_{2})} \cup \widehat{\mathcal F}_{(d_{1},d_{2}-1)}$
can have $\chi_3(F(x_{q+1})) = \varepsilon_{q+1}$. Also, the number of zeros among $\varepsilon_1,
\dots, \varepsilon_q$ is now $m-1$. Thus using \eqref{useful}
\begin{eqnarray} \label{case1section3}\nonumber
&&\left|\left\{ F \in \widehat{\mathcal F}_{[d_{1},d_{2}]} :
\chi_3(F(x_i))=\varepsilon_i, 1 \leq i \leq q+1 \right\}\right|
\\   \nonumber &&=
\sum_{\alpha \in \F_q^*} \left |\left\{ F \in {\mathcal F}_{(d_{1}-1,d_{2})}
\cup {\mathcal F}_{(d_{1},d_{2}-1)} :
\chi_3(F(x_i))=\varepsilon_i \chi_3^{-1}(\alpha), 1 \leq i \leq q
\right\}\right|
\\ \nonumber
&& = 2 (q-1) \left( \frac{ K q^{d_1+d_2-1}}{\zeta_q(2)^2}  \left(
\frac{2}{q+2} \right)^{m-1}  \left( \frac{q}{3(q+2)} \right)^{q-m+1}
\right)
\\
&&\times \left(1+ O\left( q^{ -(1-\varepsilon) (d_2-m)+\varepsilon q}+q^{-(d_1-m)/2+q} \right)\right).
\end{eqnarray}

\textbf{Case 2:} $\varepsilon_{q+1}=1, \omega,$ or $\omega^2$.

In this case, only polynomials from $\widehat{\mathcal
F}_{(d_{1},d_{2})}$ can have  $\chi_3(F(x_{q+1})) = \varepsilon_{q+1}$, and there are $m$ values of
$\varepsilon_1, \dots, \varepsilon_q$ which are zero. Thus
\begin{eqnarray} \label{case2section3}\nonumber
&&\left|\left\{ F \in \widehat{\mathcal F}_{[d_{1},d_{2}]} :
\chi_3(F(x_i))=\varepsilon_i, 1 \leq i \leq q+1 \right\}\right| \\ \nonumber
 &&= \sum_{{\alpha \in
\F_q^*}\atop{\chi_3(\alpha)=\varepsilon_{q+1}}}\left|\left\{ F \in {\mathcal
F}_{(d_{1},d_{2})} :
\chi_3(F(x_i))=\varepsilon_i\varepsilon_{q+1}^{-1}, 1 \leq i \leq q
\right\}\right|
\\
&&= \nonumber \frac{q-1}{3} \frac{K q^{d_1+d_2}}{\zeta_q(2)^2}  \left(
\frac{2}{q+2} \right)^m \left( \frac{q}{3(q+2)} \right)^{q-m}
\\
&&\times \left(1+ O\left( q^{ -(1-\varepsilon) (d_2-m)+\varepsilon q}+q^{-(d_1-m)/2+q} \right)\right),
\end{eqnarray}
which is the same as \eqref{case1section3}.

Then, it follows from \eqref{numberpoly}, \eqref{case1section3} and \eqref{case2section3} that
\begin{eqnarray*}
\frac{\left|\left\{F \in \widehat{\mathcal{F}}_{[d_1,d_2]}  \;:\;
\chi_3(F(x_i))=\varepsilon_i, 1 \leq i \leq
q+1\right\}\right|}{\left|\widehat{\mathcal{F}}_{[d_1,d_2]}\right|} &=& \left(
\frac{2}{q+2} \right)^{m}  \left( \frac{q}{3(q+2)} \right)^{q+1-m}\\
&\times& \left(1+ O\left( q^{ -(1-\varepsilon) (d_2-m)+\varepsilon q}+q^{-(d_1-m)/2+q} \right)\right).\\
\end{eqnarray*}
Putting everything together, we obtain
\begin{eqnarray*}
&& \frac{\left|\left\{C \in \mathcal H^{(d_1,d_2)} \;:\; \Tr (\Frob_C)= -s
\right\}\right|'}{\left|\mathcal H^{(d_1,d_2)}\right|'}  = \frac{\left|\left\{F \in
\widehat{\mathcal F}_{[d_{1},d_{2}]}
  \;:\; \widehat{S}_3(F)= s \right\}\right|}{
\left|\widehat{\mathcal F}_{[d_{1},d_{2}]} \right|} \\
&&\\
&&= \sum_{{(\varepsilon_1, \dots, \varepsilon_{q+1})} \atop
{\varepsilon_1 + \cdots + \varepsilon_{q+1}= s}} \frac{\left|\left\{F \in
\widehat{\mathcal{F}}_{[d_1,d_2]}   \;:\;
\chi_3(F(x_i))=\varepsilon_i, 1 \leq i \leq
q+1\right\}\right|}{\left|\widehat{\mathcal{F}}_{[d_1,d_2]}\right|} \\
&&= \sum_{{(\varepsilon_1, \dots, \varepsilon_{q+1})} \atop
{\varepsilon_1 + \cdots + \varepsilon_{q+1}= s}} \left( \frac{2}{q+2}
\right)^{m}  \left( \frac{q}{3(q+2)} \right)^{q+1-m}\left(1+ O\left( q^{ -(1-\varepsilon) (d_2-m)+\varepsilon q}+q^{-(d_1-m)/2+q} \right)\right)
\\
&&= \Prob \left( \sum_{i=1}^{q+1} X_i =  s \right)\left(1+ O\left( q^{ -(1-\varepsilon) (d_2-q)+\varepsilon q}+q^{-(d_1-q)/2+q} \right)\right)
\end{eqnarray*}
where $X_1, \ldots, X_{q+1}$ $q+1$ are i.i.d.\ random variables that
take the value $0$ with probability $2/(q+2)$ and $1, \omega, \omega^2$
each with probability $q/(3(q+2)).$
\end{proof}
We concentrate on the proof of Theorem \ref{countingF} in the next
section.

\section{Distribution of the trace for cube-free polynomials}
\label{generalcount}

In this section we prove Theorem \ref{countingF} by obtaining asymptotic formulas for $|\mathcal F_{(d_1,d_2)}|$  and \\$\left|\left\{ F \in \mathcal{F}_{(d_1,d_2)} : F(x_i) = a_i,
1 \leq i \leq q \right\}\right| $ for $d_1, d_2\rightarrow \infty$. We begin with two lemmas that count the number of polynomials that obtain specified nonzero values and are relatively prime to a fixed polynomial.

\begin{lemma}\label{lemma-V}
For  $\min\{d, q \} \geq\ell \geq 0$ let $x_1, x_2, \dots, x_{\ell} \in \F_q$ be distinct elements. Let $U \in \F_q[X]$ be such that $U(x_i)\neq 0$ for $i=1, \dots, \ell$. Let $a_1, a_2, \dots, a_{\ell} \in \F_q^*$, then
\begin{equation*}
\left|\left\{F \in V_d : (F,U)=1, F(x_i)=a_i, 1 \leq i \leq \ell \right\}\right|=q^{d-\ell} \prod_{P | U} (1-q^{-\deg P}).
\end{equation*}
Note that when $\ell=0$, there is no condition imposed at any point in $\mathbb F_q$.
\end{lemma}

\begin{proof}
By inclusion-exclusion we have
\begin{equation*}
\left|\left\{F \in V_d : (F,U)=1, F(x_i)=a_i, 1 \leq i \leq \ell \right\}\right|=\sum_{D | U} \mu(D) \sum_{{{F\in V_d}\atop{D| F}}\atop{F(x_i)=a_i}} 1,
\end{equation*}
where $\mu$ is the M\"oebius function. Using the fact that $U(x_i)\neq 0$ this equals
\begin{align*}
 \sum_{D | U} \mu(D) \sum_{{G \in V_{d-\deg D}}\atop{G(x_i)=a_i D(x_i)^{-1}}}1.
\end{align*}
By Lemma \ref{KRlemma4}  this equals
\[
\sum_{D | U} \mu(D)  q^{d-\deg D-\ell} =q^{d-\ell} \sum_{D| U} \mu(D) q^{-\deg D}.
\]
The function $f(D)=\mu(D)q^{-\deg D}$ is multiplicative, so is $g(U)=\sum_{D|U}f(D)$, and
\[
g(P^e)=1-q^{-\deg P}
\]
for $e\geq 1$.
Applying this to the last equation,
\[
\left|\left\{F \in V_d : (F,U)=1, F(x_i)=a_i, 1 \leq i \leq \ell \right\}\right|=q^{d-\ell} \prod_{P | U} (1-q^{-\deg P}).
\]
\end{proof}

Our proof of the next lemma follows the same steps as the proof of Lemma \ref{KRlemma5} (Lemma 5 in \cite{KR}), with the added condition that the polynomials we are counting are relatively prime to a fixed polynomial.

\begin{lemma} \label{lemma-T}
 For $q\geq \ell \geq 0$ let $x_1, \ldots, x_{\ell} \in \mathbb F_q$ be distinct elements. Let $U \in \F_q[X]$ be such that $U(x_i)\neq 0$ for $i=1, \dots, \ell$.  Let $a_1, \ldots, a_\ell \in \mathbb F_q^*$. Let $S_{d}^U(\ell)$ be the number of elements in the set
\[  \left\{ F \in \mathcal{F}_d: (F, U)=1, \; F(x_i)=a_i, \;
1 \leq i \leq \ell \right\}. \]
Then
\[\
S_{d}^U(\ell) =\frac{q^{d-\ell}}{\zeta_q(2) (1-q^{-2})^{\ell}}
\prod_{P \mid U} (1 + q^{-\deg{P}})^{-1} + O \left( q^{d/2} \right).
\]
\end{lemma}
\begin{proof}
By inclusion-exclusion we have
\begin{eqnarray*}
S_d^U(\ell)=\sum_{{D,\, \deg D \leq d/2}\atop{(D, U)=1}} \mu(D) \left| \left\{F\in V_{d-2 \deg(D)} : (F, U)=1,  D(x_i)^2F(x_i)=a_i, 1 \leq i \leq \ell\right\}\right|.
\end{eqnarray*}
We denote by $\tilde\sum$ the sum over all polynomials $D$ such that  $D(x_i)\neq 0$ for $1 \leq i \leq \ell$. Then
\begin{eqnarray*}
S_d^U(\ell)=\tilde{\sum _{{\deg D \leq d/2}\atop{ (D, U)=1}}} \mu(D) \left| \left\{F\in V_{d-2 \deg(D)} : (F, U)=1, F(x_i)=a_iD(x_i)^{-2}, 1 \leq i \leq \ell\right\}\right|.
\end{eqnarray*}
For $d-2\deg D \geq \ell$, by Lemma \ref{lemma-V} we have
\[
\left | \left\{F\in V_{d-2 \deg D} : (F, U)=1, F(x_i)=a_iD(x_i)^{-2}, 1 \leq i \leq \ell \right\}\right| =q^{d-2\deg D-\ell} \prod_{P | U} (1-q^{-\deg P}).
\]
Therefore
\begin{equation}\label{eqn-S_d}
S_d^U(\ell)= q^{d-\ell}\prod_{P | U} (1-q^{-\deg P})\tilde{ \sum_{{\deg D < (d-\ell)/2 }\atop{(D,U)=1}} }\mu(D)q^{-2 \deg D} + \text{ Error}.
\end{equation}
There is at most one $F$ of degree less than or equal to $\ell$ that takes $\ell$ prescribed values at $\ell$ distinct points, thus
\begin{equation}\label{eqn:error}
\text{Error } \ll \sum_{(d-\ell)/2 \leq \deg D \leq d/2} 1 = q^{d/2}\left(\frac{1-q^{-\ell/2-1}}{1-q^{-1}} \right)=O(q^{d/2}).
\end{equation}
Now we observe that
\begin{equation}\label{eqn:sum}
\tilde{\sum_{{\deg D < (d-\ell)/2}\atop{(D,U)=1}}}  \mu(D)q^{-2 \deg D} = \tilde{\sum_{{D}\atop{(D,U)=1} }  }\mu(D)q^{-2 \deg D} +O(q^{(\ell-d)/2})
\end{equation}
and
\[
\tilde{\sum_{{D} \atop{(D,U)=1}}} \mu(D)|D|^{-2s}=\prod_{{P, P(x_i)\neq 0}\atop{P \nmid U}} (1-|P|^{-2s}).
\]
Using that $U(x_i)\neq 0$, we find that
\begin{equation}\label{eqn:sum2}
\tilde{\sum_{{D} \atop{(D,U)=1}}} \mu(D)|D|^{-2s}=\frac{1}{\zeta_q(2s) (1-q^{-2s})^\ell} \prod_{P | U} (1-|P|^{-2s})^{-1}.
\end{equation}
By \eqref{eqn-S_d}, \eqref{eqn:error},     \eqref{eqn:sum} and \eqref{eqn:sum2} we have
\begin{eqnarray*}
&&S_d^U(\ell)= q^{d-\ell}\prod_{P | U} (1-q^{-\deg P}) \left( \frac{1}{\zeta_q(2) (1-q^{-2})^\ell} \prod_{P | U} (1-|P|^{-2})^{-1}+O(q^{(\ell-d)/2}) \right)
+ O(q^{d/2})
\\
&&=\frac{q^{d-\ell}}{\zeta_q(2) (1-q^{-2})^{\ell}}
\prod_{P \mid U} (1 + q^{-\deg{P}})^{-1} + O \left( q^{d/2} \right).
\end{eqnarray*}
\end{proof}

We now use Lemma \ref{lemma-T} along with the function field version of the Tauberian Theorem to count the number of polynomials in $\mathcal F_{(d_1, d_2)}$ that take a prescribed set of nonzero values on $\ell$ points.
\begin{proposition} \label{nonzerovalues} Let $0 \leq \ell \leq q$, let $x_1, \dots, x_\ell$ be
distinct elements of $\F_q$, and $a_1, \dots, a_\ell \in \F_q^*$.
Then for any $1>\varepsilon > 0$, we have
\begin{eqnarray*} \left|\left\{F \in \mathcal{F}_{(d_1,d_2)}: F(x_i)=a_i,
1\leq i \leq \ell\right\}\right|=
\frac{Kq^{d_1+d_2}}{\zeta_q(2)^2 }\left(
\frac{q}{(q+2)(q-1)}\right)^\ell \left(1+
O\left(q^{-(1-\varepsilon)d_2+\varepsilon \ell}+ q^{-d_1/2+\ell }\right)\right)
\end{eqnarray*}
where
\begin{eqnarray}
\label{constantC} K =  \prod_{P} \left( 1 - \frac{1}{(|P|+ 1)^2}
\right),
\end{eqnarray}
and the product runs over all monic irreducible polynomials of
$\F_q[X]$.

In particular, we have
\begin{eqnarray} \label{calculF}
\left| \mathcal{F}_{(d_1, d_2)} \right| &=& \frac{K
q^{d_1+d_2}}{\zeta_q(2)^2} \left( 1 + O \left(
q^{-(1-\varepsilon) d_2} + q^{- d_1/2}\right) \right).
\end{eqnarray}
\end{proposition}
\begin{proof}
First we observe that
\begin{eqnarray*}
{\left|\left\{ F \in \mathcal{F}_{(d_1,d_2)} : F(x_i) = a_i, 1
\leq i \leq \ell \right\}\right|} &=& \sum_{{F_2 \in \mathcal{F}_{d_2}}\atop {F_2(x_i)\neq 0, 1 \leq i \leq \ell}} \sum_{{{F_1 \in
\mathcal{F}_{d_1}}\atop{ F_1(x_{i})=a_i F_2(x_i)^{-2}, 1 \leq i \leq \ell } }\atop{(F_1,F_2)=1}}
1
\\
&&= \sum_{{F \in \mathcal{F}_{d_2}}\atop {F_2(x_i)\neq 0, 1 \leq i \leq \ell}} S_{d_1}^{F_2}(\ell).
\end{eqnarray*}
Using Lemma \ref{lemma-T} we have that
\begin{eqnarray*}\label{eq:fk}
&&{\left|\left\{ F \in \mathcal{F}_{(d_1,d_2)} \;:\; F(x_i) = a_i,
\; 1
\leq i \leq \ell \right\}\right|} \nonumber
\\&& =  \frac{q^{d_1-\ell}}{\zeta_q(2) (1-q^{-2})^\ell}\sum_{{F_2
\in \mathcal{F}_{d_2}}\atop{ F_2(x_i)\neq 0, 1 \leq i \leq \ell}} \prod_{P|F_2}(1+q^{-\deg P})^{-1}
+  \sum_{{F_2
\in \mathcal{F}_{d_2}}\atop{ F_2(x_i)\neq 0, 1 \leq i \leq \ell}} O \left( q^{d_1/2}\right).
\end{eqnarray*}
Then by Lemma \ref{KRlemma3} and Lemma \ref{KRlemma5},
\begin{equation} \label{firstformula2}
\nonumber {\left|\left\{ F \in \mathcal{F}_{(d_1,d_2)} \;:\; F(x_i) = a_i,
\; 1
\leq i \leq \ell \right\}\right|}
= \frac{q^{d_1-\ell}}{\zeta_q(2) (1-q^{-2})^\ell} \sum_{\deg F=d_2} b(F)
 +  O \left( q^{d_2 +d_1/2}\right).
\end{equation}
where for any polynomial $F$
\begin{equation}\label{b's}
b(F) =\begin{cases}  \mu^2(F)\prod_{P \mid F} (1
+ |P|^{-1})^{-1} &  F(x_i)\neq 0, 1 \leq i \leq \ell,
\\
0 & \text{otherwise.}
\end{cases}
\end{equation}
To evaluate $\sum_{\deg{F}=d_2} b(F)$, we consider
the Dirichlet series
\begin{align*}
& G(s)=\sum_F \frac{b(F)}{|F|^s}= \prod_{{P}\atop{P(x_i)\neq 0, 1\leq i \leq \ell}} \left( 1 +
\frac{1}{|P|^s} \cdot \frac{|P|}{|P|+1} \right)
\\
&=\frac{\zeta_q(s)}{\zeta_q(2s)}  H(s)\left( 1+\frac{1}{q^{s-1}(q+1)}\right)^{-\ell},
\end{align*}
where
\[
H(s)=\prod_{P} \left( 1 - \frac{1}{(|P|^s
+ 1)(|P|+1)} \right).
\]
Notice that $H(s)$ converges absolutely for Re$(s)>0$, and $G(s)$ is
meromorphic for Re$(s)>0$ with simple poles at the points $s$ where
$\zeta_q(s)=(1-q^{1-s})^{-1}$ has poles, that is, $s_n=1+i
\frac{2\pi n}{\log q}$, with $n \in \mathbb Z$. Notice that
$H(1)=K$, where $K$ is the constant given in (\ref{constantC}) and
$\res_{s=1}\zeta_q(s)=\frac{1}{\log q}$.
Thus $G(s)$ has a simple pole at $s=1$ with residue
\[
\frac{K}{\zeta_q(2)\log q}\left( \frac{q+1}{q+2}\right)^\ell.
\]
Using Theorem 17.1 in \cite{Ro}, which is the function field version
of the Wiener-Ikehara Tauberian Theorem, we get that
\begin{eqnarray}\label{tauberian2}
\sum_{\deg F=d_2} b(F) = \frac{K}{\zeta_q(2)}\left( \frac{q+1}{q+2}\right)^\ell q^{d_2}+O_q(q^{\varepsilon d_2}).
\end{eqnarray}
We remark that it is important for Theorem \ref{momentsthm} and Corollary \ref{momentscor} to get an error term which is independent of $q$. From the proof of Theorem 17.1 in \cite{Ro}, one sees that the hidden constant in the error term of \eqref{tauberian2} is bounded by
\begin{eqnarray*}
&&\max_{|q^{-s}|=q^{-\varepsilon}}\left|H(s)\right|\left|\left(1+\frac{1}{q^{s-1}(q+1)}\right)^{-\ell}\right| \ll  \left(1-q^{-\varepsilon}\right)^{-\ell}
= \left(\frac{q^{\varepsilon}}{q^{\varepsilon}-1} \right)^\ell \ll q^{\varepsilon \ell}.
\end{eqnarray*}

Thus we have
\begin{eqnarray}\label{tauberian3}
\sum_{\deg{F}=d_2} b(F) = \frac{K}{\zeta_q(2)}
\left(\frac{q+1}{q+2}\right)^\ell q^{d_2} +
O\left(q^{\varepsilon(d_2+\ell)}\right).
\end{eqnarray}

Replacing \eqref{tauberian3} in \eqref{firstformula2}, we get
\begin{eqnarray*}
&&{\left|\left\{ F \in \mathcal{F}_{(d_1,d_2)} \;:\; F(x_i) = a_i,
\; 1
\leq i \leq \ell\right\}\right|} \nonumber \\
&&=
\frac{Kq^{d_1+d_2}}{\zeta_q(2)^2 }\left(
\frac{q}{(q+2)(q-1)}\right)^\ell  \left(1+
O\left(q^{-(1-\varepsilon)d_2+\varepsilon \ell}+ q^{-d_1/2+\ell }\right)\right).
\end{eqnarray*}
\end{proof}

Before we obtain the number of $\mathcal F_{(d_1, d_2)}$ that take any set of prescribed (zero or nonzero) values, we first need an intermediary step involving the number of zeros in $F_2$.
We recall that  $F \in \mathcal{F}_{(d_1,d_2)}^k$ is the set of
monic polynomials $F=F_1 F_2^2  \in \mathcal{F}_{(d_1,d_2)}$ such
that $F_2$ has exactly $k$ zeros over $\F_q$.

\begin{corollary}\label{whatever}
Let $x_1, \ldots, x_{q}$  be an enumeration of elements in $\mathbb F_q$. Let $a_{1}= \ldots = a_{m} = 0$, and $a_{m+1}, \ldots, a_q \in \mathbb F_q^*$.
Then, for $\varepsilon >0$
\begin{eqnarray*}
{\left|\left\{ F \in \mathcal{F}_{(d_1,d_2)}^k \;:\; F(x_i) = a_i,
\; 1
\leq i \leq q \right\}\right|} &=&\binom{m}{k}
\frac{Kq^{d_1+d_2}}{\zeta_q(2)^2 }\left( \frac{1}{q+2}\right)^m\left(
\frac{q}{(q+2)(q-1)}\right)^{q-m}
\\
&\times& \left(1+
O\left(q^{-(1-\varepsilon)(d_2-k)+\varepsilon q}+ q^{-(d_1+k-m)/2+ q }\right)\right).
\end{eqnarray*}
\end{corollary}

\begin{proof}

The $k$ roots of $F_2$ must be among the
$x_{\ell+1}, \dots, x_{\ell+m}$, and the remaining elements of $x_{\ell+1}, \dots x_{\ell+m}$ must be roots of $F_1$. Thus
we can write
\[
F(x)= \prod_{j=1}^k(x-x_{i_j})^2 \prod_{v=1}^{m-k}(x-x_{i_v} )G(x),
\]
with $G(x) \in
\mathcal{F}_{(d_1-m+k, d_2-k)}$, $G(x_i) \neq 0$ for $\ell+1\leq i \leq \ell+m$ and $G(x_i)=a_i  \prod_{j=1}^k(x_i-x_{i_j})^{-2} \prod_{v=1}^{m-k}(x_i-x_{i_v} )^{-1}$ for $1 \leq i \leq \ell$. Thus,
\begin{eqnarray*}
&& {\left|\left\{ F \in \mathcal{F}_{(d_1,d_2)}^k : F(x_i) = a_i, 1
\leq i \leq \ell+m \right\}\right|}
\\
&&= \sum_{{\{i_1, \dots, i_k\}}\atop{ \subset \{\ell+1, \dots, \ell+m\}}}
\sum_{{(\alpha_1, \dots, \alpha_m)}\atop{\in (\F_q^*)^m}} \left| \left\{ G \in
\mathcal{F}_{(d_1-m+k, d_2-k)}:   \; G(x_i)=a_i  \prod_{j=1}^k(x_i-x_{i_j})^{-2} \prod_{v=1}^{m-k}(x_i-x_{i_v} )^{-1},   \right.\right.\\
&&\left. \left. \hspace{2.1in} 1 \leq i \leq \ell,    \; G(x_{\ell+i})=\alpha_i, 1\leq i \leq m\right\} \right|.
\end{eqnarray*}
By Proposition \ref{nonzerovalues} this equals
\begin{eqnarray*}
&&\binom{m}{k}
\frac{Kq^{d_1+d_2}}{\zeta_q(2)^2 }\left( \frac{1}{q+2}\right)^m\left(
\frac{q}{(q+2)(q-1)}\right)^{\ell} \left(1+
O\left(q^{-(1-\varepsilon)(d_2-k)+\varepsilon q}+ q^{-(d_1+k-m)/2+q }\right)\right).
\end{eqnarray*}
\end{proof}

\begin{corollary} \label{formoments} Let $x_1, \dots, x_q$ be the elements of $\F_q$, and let $a_1, \dots, a_q \in \F_q$
such that $m$ of the $a_i$ are $0$. Then for $\varepsilon >0$
\begin{eqnarray*}
\left|\left\{ F \in \mathcal{F}_{(d_1,d_2)} : F(x_i) = a_i,
1 \leq i \leq q \right\}\right| &=&\frac{K q^{d_1+d_2}}{\zeta_q(2)^2}
 \left(\frac{2}{q+2}\right)^m\left(\frac{q}{(q+2)(q-1)}\right)^{q-m} \\
 &&\times\left(1 +   O \left( q^{
-(1-\varepsilon) (d_2-m) +\varepsilon q}+
q^{ -(d_1-m)/2+q }\right)\right)
\end{eqnarray*}
and
\begin{eqnarray*}\nonumber
\frac{\left|\left\{ F \in \mathcal{F}_{(d_1,d_2)} \;:\; F(x_i) = a_i, \;
1 \leq i \leq q \right\}\right|}{\left|\mathcal{F}_{(d_1,d_2)} \right|}
&=&  \left( \frac{2}{q+2} \right)^m \left( \frac{q}{(q+2)(q-1)}
\right)^{q-m}\\
&\times&\left(1 +   O \left( q^{
-(1-\varepsilon) (d_2-m) +\varepsilon q}+
q^{ -(d_1-m)/2+q }\right)\right).
\end{eqnarray*}
\end{corollary}
\begin{proof}
We sum over $k$ in  Corollary \ref{whatever}.
\begin{eqnarray*}
&&{\left|\left\{ F \in \mathcal{F}_{(d_1,d_2)} \;:\; F(x_i) = a_i,
\; 1
\leq i \leq q \right\}\right|} \nonumber \\
&& =
\frac{Kq^{d_1+d_2}}{\zeta_q(2)^2 } \left( \frac{1}{q+2}\right)^m\left(
\frac{q}{(q+2)(q-1)}\right)^{q-m}
\sum_{k=0}^m\binom{m}{k}
 \left(1+ O\left(q^{-(1-\varepsilon)(d_2-k)+\varepsilon q}+ q^{-(d_1+k-m)/2+q }\right)\right)
\\
&& =\frac{K q^{d_1+d_2}}{\zeta_q(2)^2}\left(\frac{2}{q+2}\right)^m\left(\frac{q}{(q+2)(q-1)}\right)^{q-m} \left(1 +   O \left( q^{
-(1-\varepsilon) (d_2-m) +\varepsilon q}+q^{ -(d_1-m)/2+q }
\right)\right).
\end{eqnarray*}
For the second identity, we divide by \eqref{calculF}.
\end{proof}

To complete the proof of Theorem \ref{countingF} we note that if $\varepsilon \in \{1, \omega, \omega^2\}$, there are $\frac{q-1}{3}$ elements $\alpha \in \mathbb F_q$ such that $\chi_3(\alpha)=\varepsilon$.

 \section{Moments} \label{sectionmoments}

In this section,  we compute the moments of $\Tr (\Frob_C|_{H^1_{\chi_3}})/\sqrt{q+1}$ and prove Theorem \ref{momentsthm}. Our proof follows the same steps as the proof of the equivalent result (for the case of hyperelliptic curves) in \cite{KR}.

\begin{proof}[Proof of Theorem \ref{momentsthm}.]
Working as in Section \ref{geometry}, we first rewrite
\begin{eqnarray*} M_{j,k}(q, (d_1, d_2))
&=&\frac{1}{\left|\widehat{\mathcal{F}}_{[d_1, d_2]}\right|} \sum_{F \in
\widehat{\mathcal{F}}_{[d_1, d_2]}} \left(
\frac{\widehat{S}_3(F)}{\sqrt{q+1}} \right)^j \left(
\frac{\overline{\widehat{S}_3(F)}}{\sqrt{q+1}} \right)^k.
\end{eqnarray*}
Since every $F \in \widehat{\mathcal{F}}_{[d_1, d_2]}$ can be
written uniquely as $F= \alpha G$ for some $\alpha \in \F_q^*$ and $G \in
{\mathcal{F}}_{[d_1, d_2]}$, and
\[
\widehat{S}_3(F) = \chi_3(\alpha) \sum_{x \in \mathbb{P}^1(\F_q)}
\chi_3(G(x_i)) = \chi_3(\alpha) \widehat{S}_3(G),
\]
we have
\begin{eqnarray*}
   M_{j,k}(q, (d_1, d_2))
&=&  \frac{1}{\left|\widehat{\mathcal{F}}_{[d_1, d_2]}\right|} \sum_{\alpha \in
\F_q^*} \chi_3(\alpha)^{j-k} \sum_{F \in {\mathcal{F}}_{[d_1, d_2]}}
\left( \frac{\widehat{S}_3(F)}{\sqrt{q+1}} \right)^j\left(
\frac{\overline{\widehat{S}_3(F)}}{\sqrt{q+1}} \right)^k \\
&=& \frac{1}{\left|{\mathcal{F}}_{[d_1, d_2]}\right|}  \sum_{F \in
{\mathcal{F}}_{[d_1, d_2]}} \left(
\frac{\widehat{S}_3(F)}{\sqrt{q+1}} \right)^j \left(
\frac{\overline{\widehat{S}_3(F)}}{\sqrt{q+1}} \right)^k\end{eqnarray*}
when $j\equiv k \,(\mathrm{mod}\, 3)$
and  $M_{j,k}(q, (d_1, d_2)) = 0$ when $j \not\equiv k \,(\mathrm{mod}\,
3)$.

Assume from now on that $j \equiv k \,(\mathrm{mod}\, 3)$. Then,
\begin{eqnarray*}
   M_{j,k}(q, (d_1, d_2))
&=& \frac{(q+1)^{-(j+k)/2}}{\left|{\mathcal{F}}_{[d_1, d_2]}\right|}  \sum_{F \in
{\mathcal{F}}_{[d_1, d_2]}} \!\!\left( \sum_{x \in \mathbb{P}^1(\F_q)}
\chi_3(F(x)) \right)^{j} \!\!\left( \sum_{x \in \mathbb{P}^1(\F_q)}
\overline{\chi_3(F(x))} \right)^{k} \\
&=&\frac{(q+1)^{-(j+k)/2}}{\left|{\mathcal{F}}_{[d_1, d_2]}\right|} \sum_{{x_1, \dots,
x_j \in \mathbb{P}^1(\F_q)}\atop{y_1, \dots,
y_k \in \mathbb{P}^1(\F_q)}} \sum_{F \in {\mathcal{F}}_{[d_1, d_2]}}
\chi_3(F(x_1) \dots F(x_j))\overline{\chi_{3}(F(y_1) \dots F(y_k))}
\end{eqnarray*}
\begin{equation} \label{explodedmoments}
= (q+1)^{-(j+k)/2} \sum_{\ell = 1}^j\sum_{m = 1}^k d(j,k,\ell,m)\!\!\!
\sum_{({\bf x}, {\bf y}, {\bf b}, {\bf c}) \in P_{j,k,\ell,m}}
\frac{1}{\left|{\mathcal{F}}_{[d_1, d_2]}\right|} \sum_{F \in
{\mathcal{F}}_{[d_1, d_2]}} \prod_{i=1}^\ell \chi_3(F(x_i))^{b_i}\prod_{i=1}^m \overline{\chi_3(F(y_i))^{c_i}}
\end{equation}
where \begin{eqnarray*} P_{j,k,\ell,m} &=& \left\{(\mathbf x,\mathbf y, \mathbf
b, \mathbf c)\;:\; \mathbf x= (x_1, \ldots, x_{\ell})  \in \mathbb P^1(\mathbb
F_q)^\ell, x_i \textrm{'s distinct}, \mathbf y = (y_1, \ldots, y_m) \in \mathbb P^1(\mathbb
F_q)^m  \right.\\
&&\left . y_i \textrm{'s distinct}, \mathbf b = (b_1, \ldots, b_\ell) \in \mathbb Z_{>0}^\ell, \mathbf c = (c_1, \ldots, c_m) \in \mathbb Z_{>0}^m,\sum_{i=1}^\ell
b_i =j, \sum_{i=1}^m
c_i =k \right\},
\end{eqnarray*} and $d(j,k,\ell,m)$ is a certain combinatorial factor. We
do not need exact formulas for the $d(j,k,\ell,m)$, but note that
\begin{eqnarray*} 
\sum_{\ell=1}^j \sum_{m=1}^k d(j,k,\ell,m)  \sum_{({\mathbf x}, {\mathbf y},{\mathbf  {b}}, {\mathbf c}) \in
P_{j,k,\ell,m} }  1 =   (q+1)^{j+k} . \end{eqnarray*}

We now fix a vector $(\mathbf x, \mathbf y, \mathbf b,\mathbf c) \in P_{j, k,\ell,m},$ and we
compute
\begin{eqnarray*}
\frac{1}{\left|{\mathcal{F}}_{[d_1, d_2]}\right|} \sum_{F \in
{\mathcal{F}}_{[d_1, d_2]}} \prod_{i=1}^\ell \chi_3(F(x_i))^{b_i} \prod_{i=1}^m \overline{\chi_3(F(y_i))^{c_i}}.
\end{eqnarray*}
Suppose that  $\{x_1, \dots, x_\ell,y_1, \dots, y_m\}=\{z_1, \dots, z_h\}$, in other words, that $\mathbf{x}$ and $\mathbf{y}$ have $\ell+m-h$ coordinates in common. To simplify the notation,  we will denote by $f_i$ the corresponding exponent for $z_i$ which is equal to some $b_i$, $c_i$ or $b_i-c_i$ depending on whether the value $z_i$ appears in $\{x_1, \dots, x_\ell\}$, $\{y_1, \dots, y_m\}$, or in both sets. We also adopt the convention that $f_i$ could be equal to $0$, in which case $\chi_{3}(F(z_i))^{f_i}=1$ if $F(z_i) \neq 0$ and $\chi_{3}(F(z_i))^{f_i}=0$  if $F(z_i)=0$. With this notation, we want to compute
\begin{eqnarray} \label{tocompute}
\frac{1}{\left|{\mathcal{F}}_{[d_1, d_2]}\right|} \sum_{F \in
{\mathcal{F}}_{[d_1, d_2]}} \prod_{i=1}^h \chi_3(F(z_i))^{f_i}.
\end{eqnarray}

There are two cases.

{\bf Case 1:} Suppose that $z_t$ is the point at infinity, for some
$1\leq t \leq h.$ Then, only polynomials in ${\mathcal
F}_{(d_1,d_2)}$ have a nonzero contribution to (\ref{tocompute}) and
$\chi_3(F(z_t))^{f_t} = 1$. This gives that
\begin{eqnarray*} 
\sum_{F \in
{\mathcal{F}}_{[d_1, d_2]}} \prod_{i=1}^h \chi_3(F(z_i))^{f_i}&=& \sum_{{a_i \in \F_q^*}\atop{1\leq i \leq
h, i \neq t}} \prod_{{i=1}\atop{i \neq t}}^h \chi_3(a_i)^{f_i}
\sum_{{F \in {\mathcal{F}}_{(d_1, d_2)}}\atop{F(z_i)=a_i, 1 \leq i
\leq h, i \neq t}} 1 .
\end{eqnarray*}
Suppose that all $f_i$ for $1 \leq i \leq h$ and $i \neq t$ are
multiples of 3. Then, using Proposition \ref{nonzerovalues}, we get
\begin{eqnarray} \label{bi0mod3}
\sum_{F \in {\mathcal{F}}_{(d_1, d_2)}} \prod_{i=1}^h
\chi_3(F(z_i))^{f_i}&=& \frac{K q^{d_1+d_2}}{\zeta_q(2)^2} \left(
\frac{q}{q+2} \right)^{h-1}  \left( 1 + O \left(
q^{-(1-\varepsilon)d_2 + \varepsilon (h-1)}+
q^{-d_1/2+h-1} \right) \right).
\end{eqnarray}

Suppose that there exists a $f_i$ with $1 \leq i \leq h$ and $i
\neq t$ such that $f_i$ is not a multiple of 3. Without loss of
generality, suppose that $f_1 \equiv 1 \,(\mathrm{mod}\, 3)$ and $t
\neq 1$. Then, using again Proposition \ref{nonzerovalues}, we get
\begin{eqnarray} \nonumber
\sum_{F \in {\mathcal{F}}_{[d_1, d_2]}} \prod_{i=1}^h
\chi_3(F(z_i))^{f_i} &=& \sum_{{{a_i \in \F_q^*}\atop{1\leq i \leq
h, i \neq t}}\atop{\chi_3(a_1)=1}}  \prod_{{i=2}\atop{i \neq
t}}^h \chi_3(a_i)^{f_i} \sum_{{F \in {\mathcal{F}}_{(d_1,
d_2)}}\atop{F(z_i)=a_i, 1 \leq i \leq h, i \neq t}} 1  \\
\nonumber&& \hspace{-1.6in} + \omega \sum_{{{a_i \in \F_q^*}\atop{1\leq i \leq
h, i \neq t}}\atop{\chi_3(a_1)=\omega}}  \prod_{{i=2}\atop{i \neq
t}}^h \chi_3(a_i)^{f_i} \sum_{{F \in {\mathcal{F}}_{(d_1,
d_2)}}\atop{F(z_i)=a_i, 1 \leq i \leq h, i \neq t}} 1   + \omega^2 \sum_{{{a_i \in \F_q^*}\atop{1\leq i \leq
h, i \neq t}}\atop{\chi_3(a_1)=\omega^2}}  \prod_{{i=2}\atop{i \neq
t}}^h \chi_3(a_i)^{f_i} \sum_{{F \in {\mathcal{F}}_{(d_1,
d_2)}}\atop{F(z_i)=a_i, 1 \leq i \leq h, i \neq t}} 1  \\
\label{binot0mod3} &&\hspace{-1.6in}= \frac{K
q^{d_1+d_2}}{\zeta_q(2)^2} \left(
\frac{q}{q+2} \right)^{h-1} \left( 0 +  O \left(
q^{-(1-\varepsilon)d_2 + \varepsilon (h-1)}+
q^{-d_1/2+h-1} \right) \right).
\end{eqnarray}
We remark that $j \equiv k \,(\mathrm{mod}\, 3)$ and $f_i \equiv 0
\,(\mathrm{mod}\, 3)$ for $1 \leq i \leq h$, $i \neq t$ is
equivalent to $f_i \equiv 0 \,(\mathrm{mod}\, 3)$ for $1 \leq i \leq
h$.

Using (\ref{bi0mod3}) and (\ref{binot0mod3}), and
dividing by
\[ \left|\mathcal{F}_{[d_1, d_2]}\right| = \left(\frac{q+2}{q}\right)\frac{K q^{d_1+d_2}}{\zeta_q(2)^2}  \left(1+
O \left( q^{-(1-\varepsilon) d_2}+q^{-d_1/2}
\right)\right),\] we get that
\begin{eqnarray} \nonumber
&&\frac{1}{\left|{\mathcal{F}}_{[d_1, d_2]}\right|} \sum_{F \in
{\mathcal{F}}_{[d_1, d_2]}} \prod_{i=1}^h \chi_3(F(z_i))^{f_i}
\\\label{case1} &&= \left( \frac{q}{q+2} \right)^h
\left( \delta({\bf f}, h)+  O \left(
q^{-(1-\varepsilon)d_2 + \varepsilon (h-1)}+
q^{-d_1/2+h-1} \right)
\right),
\end{eqnarray}
where $\delta({\bf f}, h)=1$ if $f_i \equiv 0 \,(\mathrm{mod}\,
3)$ for $1 \leq i \leq h$ and 0 otherwise.

{\bf Case 2:} Suppose that $z_1, \ldots, z_h \in \mathbb F_q$.
Then any $F \in {\mathcal F}_{[d_1,d_2]}$ can contribute to
(\ref{tocompute}). Repeating the reasoning above, we get
\begin{eqnarray*}
\sum_{F \in {\mathcal{F}}_{[d_1, d_2]}} \prod_{i=1}^h
\chi_3(F(z_i))^{f_i} &=& \sum_{{a_i \in \F_q^*}\atop{1\leq i \leq
h}} \prod_{{i=1}}^h \chi_3(a_i)^{f_i} \sum_{{F \in
{\mathcal{F}}_{[d_1, d_2]}}\atop{F(z_i)=a_i, 1 \leq i \leq h}} 1.
\end{eqnarray*}

Suppose that all $f_i$ for $1 \leq i \leq h$ are multiples of 3.
Then, using Proposition \ref{nonzerovalues}, we get
\begin{eqnarray*}
\sum_{F \in {\mathcal{F}}_{[d_1, d_2]}} \prod_{i=1}^h
\chi_3(F(z_i))^{f_i} &=& \frac{K q^{d_1+d_2}}{\zeta_q(2)^2} \left(
\frac{q}{q+2} \right)^{h-1} \left( 1 + O \left(
q^{-(1-\varepsilon)d_2 + \varepsilon h}+
q^{-d_1/2+h} \right)  \right).
\end{eqnarray*}
If not all $f_i$ are multiples of 3, reasoning as in Case 1, we get
\[
\sum_{F \in {\mathcal{F}}_{[d_1, d_2]}} \prod_{i=1}^h
\chi_3(F(z_i))^{f_i} = \frac{K q^{d_1+d_2}}{\zeta_q(2)^2} \left(
\frac{q}{q+2} \right)^{h-1} \left( 0 + O \left(
q^{-(1-\varepsilon)d_2 + \varepsilon h }+
q^{-d_1/2+h} \right)  \right)
\]
and
\begin{eqnarray}\nonumber
&&\frac{1}{\left|{\mathcal{F}}_{[d_1, d_2]}\right|} \sum_{F \in
{\mathcal{F}}_{[d_1, d_2]}} \prod_{i=1}^h \chi_3(F(x_i))^{f_i}\\  \label{case2} &&
= \left( \frac{q}{q+2} \right)^h
\left( \delta({\bf f}, h) + O \left(
q^{-(1-\varepsilon)d_2 + \varepsilon h }+
q^{-d_1/2+h} \right)
\right).
\end{eqnarray}

We then have the same result for Case 1 and Case 2, and replacing
(\ref{case1}) or (\ref{case2}) in (\ref{explodedmoments}), we have
\begin{eqnarray} \nonumber
  M_{j,k}(q, (d_1, d_2)) &=& \left((q+1)^{-(j+k)/2} \sum_{\ell = 1}^j \sum_{m=1}^k d(j,k,\ell,m)
\sum_{{({\bf x}, {\bf y}, {\bf b}, {\bf c}) \in P_{j,k,\ell,m}}\atop{3 \mid f_i}} \left(
\frac{q}{q+2} \right)^h \right) \\  \label{ourmoments}  &\times& \left( 1 + O \left(
q^{-(1-\varepsilon)d_2 + \varepsilon (j+k) }+
q^{-d_1/2+(j+k)} \right)\right),
\end{eqnarray}
where $h$ and $f_i$ are understood as before.

We now compute the corresponding moment of the normalized sum of the  random
variables $X_1, \dots, X_{q+1}$, i.e.
 \begin{eqnarray*}
 &&\mathbb E
\left( \left( \frac{1}{\sqrt{q+1}} \sum_{i=1}^{q+1} X_i \right)^j\left( \frac{1}{\sqrt{q+1}} \sum_{i=1}^{q+1} \overline{X_i} \right)^k
\right)
\\
=&& \frac{1}{(q+1)^{(j+k)/2}} \sum_{\ell=1}^j \sum_{m=1}^k d(j,k,\ell,m)
\sum_{({\mathbf u},{\mathbf v}, {\mathbf  {b},{\mathbf c}}) \in A_{j,k,\ell,m} }\mathbb E \left(
X_{u_1}^{b_1} \cdots X_{u_\ell}^{b_\ell} \overline{X_{v_1}}^{c_1} \cdots \overline{X_{v_m}}^{c_m}  \right)
\end{eqnarray*}
where
\begin{eqnarray*} A_{j,k,\ell,m} &=& \left\{(\mathbf u,\mathbf v, \mathbf
b, \mathbf c)\;:\; \mathbf u= (u_1, \ldots, u_{\ell}),  1\leq u_i\leq q+1, u_i \textrm{'s distinct}, \mathbf v = (v_1, \ldots, v_m), \right.\\
&& 1\leq v_i\leq q+1, v_i \textrm{'s distinct}, \mathbf b = (b_1, \ldots, b_\ell) \in \mathbb Z_{>0}^\ell, \mathbf c = (c_1, \ldots, c_m) \in \mathbb Z_{>0}^m,\\
&&\left . \sum_{i=1}^\ell
b_i =j, \sum_{i=1}^m
c_i =k \right\}.
\end{eqnarray*}

Since
\begin{equation*}\mathbb E(X_i^b\overline{X_i}^c) = \begin{cases}
0 & b \not \equiv c \, (\mathrm{mod}\, 3) ,\\
\displaystyle \frac{1}{1+2q^{-1}} & b \equiv c \, (\mathrm{mod}\, 3)
\end{cases}\end{equation*}
and $X_1, \ldots, X_{q+1}$ are independent, we get
\begin{eqnarray} \nonumber
&&\mathbb E
\left( \left( \frac{1}{\sqrt{q+1}} \sum_{i=1}^{q+1} X_i \right)^j\left( \frac{1}{\sqrt{q+1}} \sum_{i=1}^{q+1} \overline{X_i} \right)^k
\right)
\\ \label{RVmoments}
=&& \frac{1}{(q+1)^{(j+k)/2}} \sum_{\ell=1}^j \sum_{m=1}^k d(j,k,\ell,m)
\sum_{({\mathbf u},{\mathbf v}, {\mathbf  {b},{\mathbf c}}) \in A_{j,k,\ell,m} \atop{3 \mid f_i}} \left(
\frac{q}{q+2} \right)^h.
\end{eqnarray}

Since the number of terms in the sums over the sets $P_{j, k,\ell,m}$
such that $3 \mid f_i$ and  $A_{j, k, \ell, m}$ such that $3 \mid f_i$ are
the same, comparing (\ref{ourmoments}) and (\ref{RVmoments}), we
have
\begin{eqnarray*}
  M_{j,k}(q,(d_1,d_2)) &=& \mathbb E
\left( \left( \frac{1}{\sqrt{q+1}} \sum_{i=1}^{q+1} X_i \right)^j\left( \frac{1}{\sqrt{q+1}} \sum_{i=1}^{q+1} \overline{X_i} \right)^k
\right)\\
&\times &\left( 1 + O \left(
q^{-(1-\varepsilon)d_2 + \varepsilon (j+k) }+
q^{-d_1/2+j+k} \right)\right).
\end{eqnarray*}

\end{proof}

\begin{proof}[Proof of Corollary \ref{momentscor}.]
First we study the distribution of the normalized sum of the i.i.d. random variables $(X_1+\ldots+ X_{q+1})/\sqrt{q+1}$ as $q \rightarrow \infty$. Since the $X_j$'s take complex values, we first write each of them as $X_j=A_j+ \sqrt{-1}B_j$ and identify it with the $\mathbb R^2$ vector $\left( \begin{array}{c}A_j \\B_j \end{array}\right).$

Since $\mathbb E (|X_j|^2) =(1+2q^{-1})^{-1}$ and $\mathbb E(X_j)= \mathbb E(X_j^2)=0$, we have
\[\mathbb E(A_j)=\mathbb E(B_j)= \mathbb E(A_j B_j)=0\] and \[\mathbb E(A_j^2) = \mathbb E(B_j^2) = \frac{1}{2} \mathbb E(|X_j|^2).\]
The Triangular Central Limit Theorem holds for two-dimensional vector valued random variables as long as the covariance matrix is invertible. Since for us the covariance matrix not only is invertible, but also is diagonal with nonzero diagonal entries, we obtain that
\[ \frac{1}{q+1} \sum_{j=1}^{q+1}\left( \begin{array}{c}A_j \\B_j \end{array}\right) \rightarrow \mathbf N_{\mathbb R^2} \left(0, \left( \begin{array}{cc}1/2&0\\0 & 1/2 \end{array}\right)\right), \]
the two-dimensional Gaussian with mean $0$ and covariance matrix $\left( \begin{array}{cc}1/2&0\\0 & 1/2 \end{array}\right),$ whose probability density is given by
\[ \frac{1}{\pi} e^{-x^2 - y^2} dx dy.\]

Note that this measure is invariant under multiplication by $-1$. Going back to the complex valued random variables, we obtain that, as $q$ approaches infinity, the normalized sum  $(X_1+\ldots + X_{q+1})/\sqrt{q+1}$ approaches the complex Gaussian with mean zero and variance one. The probability measure of this Gaussian is given by

\[\frac{1}{\pi} e^{-|z|^2} dz.\]

Theorem \ref{momentsthm} tells us that, as $q, d_1, d_2 \rightarrow \infty,$ the moments $M_{j,k}(q,(d_1,d_2))$ approach the moments of the complex Gaussian for all $j$ and $k.$
Since the Gaussian is invariant under the change of sign, the limiting value distribution of $\Tr (\Frob_C|_{H^1_{\chi_3}})/\sqrt{q+1}$ is the complex Gaussian distribution with mean $0$ and variance $1.$
\end{proof}

\section{Hyperelliptic curves: the geometric point of view}
\label{KRrevisited}

We now revisit the results of Kurlberg and Rudnick \cite{KR} for hyperelliptic curves from the geometric point of view. The results of this section are similar to the results of Section \ref{geometry} for the case of hyperelliptic curves.

\begin{lemma} \label{KRlemma3geom}
The number of square-free polynomials of degree $d$ is
\[\left|\widehat{\mathcal F}_d\right|=\left \{ \begin{array}{ll} {q^{d+1}}{(1-q^{-1})^2} & d\geq 2,\\\\q^{d+1}(1-q^{-1}) & d=0,1.\end{array}\right.\]
\end{lemma}
\begin{proof}
This follows from Lemma \ref{KRlemma3} since $\left|\widehat{\mathcal F}_d\right|=(q-1)\left|\mathcal F_d\right|$.
\end{proof}

\begin{proof}[Proof of Theorem \ref{KRrevisitedthm}.]

Consider the hyperelliptic curve with affine model
\[ C: \quad Y^2 = F(X),\]
where $F \in \widehat{\mathcal{F}}_d$. It has genus $g$ if and only if
$d$ is either $2g+1$ or $2g+2.$ In terms of the polynomial $F$, the trace
of the Frobenius is equal to
\[-\widehat{S}_2(F) = -\sum_{x \in \mathbb P^1(\mathbb F_q)} \chi_2(F(x)),\]
where the value of $F$ at the point at infinity is given by the
value of $X^{2g+2}F(1/X)$ at zero. By running over all $F \in \widehat{\mathcal F}_{2g+1} \cup \widehat{\mathcal F}_{2g+2} $, one counts each point in the moduli space $\mathcal H_g$ exactly  $q(q^2-1)$ times. (As usual, a point $C$ in the moduli space is counted with weight $1/|\mbox{Aut}(C)|$.) With the notation from the introduction,
 \[\frac{\left|\left\{C \in \mathcal H_g:  \Tr(\Frob_C) = -s\right\}\right|'}{\left|\mathcal H_g\right|'}  = \frac{ \left| \left\{F \in \widehat{\mathcal F}_{2g+1} \cup \widehat{\mathcal F}_{2g+2}:  \widehat S_2(F) = s \right\}\right|} {\left|\widehat{\mathcal F}_{2g+1} \right| + \left| \widehat{\mathcal F}_{2g+2}\right|}  \]

Fix $x_1, \ldots, x_{q+1}$ an enumeration of the points on $\mathbb
P^1(\mathbb F_q)$ such that $x_{q+1}$ denotes the point at infinity.
Then
\[\chi_2(F(x_{q+1}) )= \begin{cases}
0 & F \in \widehat{\mathcal F}_{2g+1},\\
1 & F \in \widehat{\mathcal F}_{2g+2} \textrm{, leading coefficient is a square in  } \mathbb F_q,\\
-1 & F \in \widehat{\mathcal F}_{2g+2} \textrm{, leading coefficient
is not a square in  } \mathbb F_q.
\end{cases}\]

Pick $(\varepsilon_1, \ldots, \varepsilon_{q+1}) \in \{ 0, \pm 1
\}^{q+1}.$ Denote $m$ the number of zeros in this $(q+1)$-tuple.
We need to evaluate the probability that the character $\chi_2$
takes exactly these values as $F$ ranges over $\widehat{\mathcal
F}_{2g+1} \cup \widehat{\mathcal F}_{2g+2}.$ Namely we will show
that the results of \cite{KR} imply that
\begin{eqnarray}
\nonumber
&& \frac{\left| \left\{F \in \widehat{\mathcal F}_{2g+1} \cup \widehat{\mathcal F}_{2g+2}:  \chi_2(F(x_i)) = \varepsilon_i, 1\leq i \leq q+1 \right\}\right| } {\left| \widehat{\mathcal F}_{2g+1} \right| + \left| \widehat{\mathcal F}_{2g+2}\right|} \\
\label{char}&&= \frac{2^{m -q-1} q^{-m}}{(1+q^{-1})^{q+1}}\left( 1+ O \left(
q^{m/2+q-g-1}\right) \right).
\end{eqnarray}

{\bf Case 1:} $\varepsilon_{q+1} = 0$.  The numbers of zeros among  $\varepsilon_1, \ldots, \varepsilon_q$ is now $m-1$. Since there are no polynomials in  $\widehat{\mathcal F}_{2g+2}$ with
$\chi_2(F(x_{q+1}))=0$,  only $\widehat{\mathcal F}_{2g+1}$ contributes. There are $q-1$ possibilities for the leading coefficient of such a polynomial and thus
\begin{eqnarray*} && \left| \left\{F \in\widehat{\mathcal F}_{2g+1} \cup \widehat{\mathcal F}_{2g+2}:  \chi_2(F(x_i)) = \varepsilon_i, 1\leq i \leq q+1 \right\}\right|
\\
&&=\sum_{\alpha \in \F_q^*}\left | \left\{F \in \mathcal F_{2g+1}:  \chi_2(F(x_i)) = \varepsilon_i \chi_2(\alpha), 1\leq i \leq q \right\}\right|.
\end{eqnarray*}
Taking into account that there are $\frac{q-1}{2}$ squares in
$\mathbb F_q$ and the same number of non-squares, and using
Lemma \ref{KRproposition6}  the above expression can be written as
\begin{multline}\label{odd}
(q-1) \left(\frac{q-1}{2} \right)^{q-m +1} \frac{ (1-q^{-1})^{m} q^{2g+1-q}}{(1-q^{-2})^{q}} \left( 1+O \left(
q^{m/2+q-g-1 }\right) \right) \\
 =  \frac{2^{m-1-q} (1-q^{-1})^{q+2} q^{2g+3-m} }{ (1-q^{-2})^{q}}   \left( 1+O \left(
q^{m/2+q-g-1 }\right) \right).
 \end{multline}

With Lemma \ref{KRlemma3geom}, we compute
\begin{eqnarray}\label{sums} \left| \widehat{ \mathcal F}_{2g+1} \right| + \left| \widehat{\mathcal F}_{2g+2}\right|  = q^{ 2 g + 3} (1-q^{-1}) (1-q^{-2}), \end{eqnarray}
and dividing \eqref{odd} by \eqref{sums}
we get \eqref{char}.

{\bf Case 2:} $\varepsilon_{q+1} = \pm 1$.

This is the complementary situation, namely there are $m$ zeros among   $\varepsilon_1, \ldots, \varepsilon_q$ and only $\widehat{\mathcal F}_{2g+2}$ contributes.  By the same argument as before, and taking into
account that there are $\frac{q-1}{2}$ leading coefficients that
would give $\varepsilon_{q+1} =1$ and the same number that would
yield $\varepsilon_{q+1}=-1,$
\begin{eqnarray*} && \left| \left\{F \in \widehat{\mathcal F}_{2g+2} \cup \widehat{\mathcal F}_{2g+2}:  \chi_2(F(x_i)) = \varepsilon_i, 1\leq i \leq q+1 \right\}\right|\\
&&= \frac{q-1}{2} \left| \left\{F \in \mathcal F_{2g+2}:  \chi_2(F(x_i)) = \varepsilon_i \varepsilon_{q+1}, 1\leq i \leq q \right\}\right|. \end{eqnarray*}

By Lemma \ref{KRproposition6}, and by taking into account the number
of squares and non-squares in $\mathbb F_q$ this equals
\begin{multline*}
 \left(\frac{q-1}{2} \right)^{q+1-m} \frac{ (1-q^{-1})^{m+1} q^{2g+2 -q}}{(1-q^{-2})^{q}}  \left( 1+O \left(
q^{m/2+q-g-1}\right) \right)  \\
 =  \frac{2^{m-1-q} (1-q^{-1})^{q+2} q^{2g+3-m} }{ (1-q^{-2})^{q}}  \left( 1+O \left(
q^{m/2+q-g-1}\right) \right),
 \end{multline*}
which is the same as \eqref{odd}. The formula
\eqref{char} follows as before.

On the other hand for $X_1, \dots, X_{q+1}$ as in Theorem \ref{KRrevisitedthm}
\begin{eqnarray}\label{probhyper}
\Prob \left( X_i = \varepsilon_i, 1\leq i\leq q+1 \right)
=  \frac{2^{m -q-1} q^{-m}}{(1+q^{-1})^{q+1}}
\end{eqnarray}
and the theorem follows by summing \eqref{char} and \eqref{probhyper} over all $(q+1)$-tuples
$(\varepsilon_1, \ldots, \varepsilon_{q+1})$ such that
$\varepsilon_1+\cdots+\varepsilon_{q+1} = s$ as done at the end of Section \ref{geometry}.

\end{proof}
We remark that the probability of hitting a certain $(q+1)$-tuple does not depend on the entry at the point we designated as the point at infinity. Therefore that point behaves the same as the affine points, which is exactly what one would expect from a geometric standpoint.

\subsection{Moments}

We want to compute the moments of $
\Tr(\Frob_C)/\sqrt{q+1}.$ Namely, denote the $k$-th moment
by

\[M_k(q,g) = \frac{1}{\left|\mathcal H_g\right|'} \sideset{}{'}{\sum}_{C \in \mathcal H_g} \left( \frac{\Tr(\Frob_C)}{\sqrt{q+1}}  \right)^k.\]

For a given  curve $C\in \mathcal H_g$, its quadratic twist $C'$ also has
genus $g$ and they are not isomorphic over $\mathbb F_q.$ So both
$\Tr(\Frob_C)$ and $\Tr(\Frob_{C'})= -\Tr(\Frob_C)$ appear in our
sum. This implies that for $k$ odd we have $M_k(q,g)=0.$

\begin{theorem}\label{thm2}
If $g$, $q$ both tend to infinity, then the moments of
$\Tr(\Frob_C)/\sqrt{q+1}$, as $C$ runs over the moduli space
$\mathcal H_g$ of hyperelliptic curves of genus $g$, are
asymptotically Gaussian with mean $0$ and variance $1$. In particular
the limiting value distribution is a standard Gaussian.
\end{theorem}

As before, by looking at curves of the form
\[ Y^2 = F(X)\]
as $F$ ranges over $\widehat{\mathcal F}_{2g+1} \cup
\widehat{\mathcal F}_{2g+2}$ we run over each point in $\mathcal
H_g$ exactly $q(q^2-1)$ times. The trace of Frobenius of the curve with the  above
affine model is given by $\widehat S_2(F).$ As a result, for $k$ even, we can write
the $k$-th moment as
\begin{eqnarray*}
M_k(q,g) &=&  (-1)^k\frac{1} {\left|\widehat{\mathcal F}_{2g+1} \cup \widehat{\mathcal F}_{2g+2}\right|}\sum_{F \in \widehat{\mathcal F}_{2g+1} \cup \widehat{\mathcal F}_{2g+2}}  \left( \frac{\widehat S_2(F)}{\sqrt{q+1}}  \right)^k \\
&&\\
& = & \frac{(q+1)^{-k/2} } {\left|\widehat{\mathcal F}_{2g+1} \cup \widehat{\mathcal F}_{2g+2}\right|}\sum_{F \in \widehat{\mathcal F}_{2g+1} \cup \widehat{\mathcal F}_{2g+2}}  \left( \sum_{x \in \mathbb P^1(\mathbb F_q)} \chi_2(F(x)) \right)^k \\
&&\\
&=& \frac{1}{(q+1)^{k/2}} \sum_{x_1, \ldots, x_k \in \mathbb P^1(\mathbb F_q) }   \frac{1} {\left|\widehat{\mathcal F}_{2g+1} \cup \widehat{\mathcal F}_{2g+2}\right|}\sum_{F \in \widehat{\mathcal F}_{2g+1} \cup \widehat{\mathcal F}_{2g+2}}\!\!\!\!\!\!  \chi_2(F(x_1)\cdots F(x_k))\\
&&\\
&=& \frac{1}{(q+1)^{k/2}} \sum_{\ell=1}^k c(k,\ell) \!\!\!\sum_{({\mathbf
x},{\mathbf  {b}}) \in P_{k,\ell} }   \frac{1} {\left|\widehat{\mathcal
F}_{2g+1} \cup \widehat{\mathcal F}_{2g+2}\right|} \sum_{F \in
\widehat{\mathcal F}_{2g+1} \cup \widehat{\mathcal F}_{2g+2}}
\!\!\!\!\!\!\!\!\! \chi_2\left( \prod_{i=1}^l F(x_i)^{b_i}\right),
\end{eqnarray*}
where
\[P_{k,\ell} = \left\{(\mathbf x, \mathbf b): \mathbf x= (x_1, \ldots, x_\ell) \in \mathbb P^1(\mathbb F_q)^\ell, x_i \textrm{'s distinct }, \mathbf b = (b_1, \ldots, b_\ell) \in \mathbb Z_{>0}^\ell, \sum_{i=1}^l b_i =k   \right\}, \]
and $c(k,\ell)$ is a certain combinatorial factor. We do not need exact
formulas for the $c(k,\ell)$, but note that
\begin{equation}\label{comb}
\sum_{l=1}^k c(k,\ell)  \sum_{({\mathbf x},{\mathbf  {b}}) \in P_{k,\ell}
}  1 =   (q+1)^ k. \end{equation}

Fix a vector $(\mathbf x, \mathbf b) \in P_{k,\ell}.$ There are two
cases.

{\bf Case 1:} $x_j$ is the point at infinity, for some $1\leq j \leq \ell.$

Then only polynomials in $ \widehat{\mathcal F}_{2g+2}$ have a
nonzero contribution and we can write
\[\sum_{F \in \widehat{\mathcal F}_{2g+2}} \chi_2\left( \prod_{i=1}^\ell F(x_i)^{b_i}\right) = \sum_{a_j \in \mathbb F_q^*} \chi_2(a_j)^{b_j} \sum_{G \in \mathcal F_{2g+2}}  \prod_{i\neq j} \chi_2(G(x_i))^{b_i}.
\]
Note that $G$ ranges over \emph{monic }polynomials of degree $2g+2$
and we can write the above expression becomes
\[ \sum_{a_j \in \F_q^*} \chi_2(a_j)^{b_j} \sum_{G \in \mathcal F_{2g+2}}  \prod_{i\neq j} \chi_2(G(x_i))^{b_i}= \sum_{a_i \in \F_q^* \atop 1\leq i \leq l} \sum_{G \in \mathcal F_{2g+2} \atop {G(x_i) = a_i, i\neq j} }\prod_{i=1}^\ell \chi_2(a_i)^{b_i}.\]

Thus, by Lemma \ref{KRlemma5}, the contribution to the moment of
such a term is
\[\begin{cases}
 O\left( q^{1+ \ell+g}\right) & \textrm{ if any of the }b_i\textrm{'s is odd},\\
 \\
\displaystyle  \frac{q^{2g+3 }(1-q^{-1})^{\ell+1}}{(1-q^{-2})^{\ell-1}} +
O\left( q^{1+\ell+g}\right) &  \textrm{ if all the }b_i\textrm{'s are
even}.
 \end{cases}\]

{\bf Case 2:} $x_1, \ldots, x_\ell \in \mathbb F_q$.

Then both  $ \widehat{\mathcal F}_{2g+1}$ and  $ \widehat{\mathcal
F}_{2g+2}$ contribute. Repeating the reasoning from before, the
contribution is
\[ O\left( q^{3/2+\ell+g}\right) \]
unless all the $b_i$'s are even. In which case
\[ \sum_{F \in \widehat{\mathcal F}_{2g+1}} \chi_2\left( \prod_{i=1}^\ell F(x_i)^{b_i}\right) = \frac{q^{2g+2} (1-q^{-1})^{\ell+2}}{(1-q^{-2})^\ell}  + O\left( q^{1+\ell+g}\right) \]
and
\[ \sum_{F \in \widehat{\mathcal F}_{2g+2}} \chi_2\left( \prod_{i=1}^\ell F(x_i)^{b_i}\right) = \frac{q^{2g+3} (1-q^{-1})^{\ell+2}}{(1-q^{-2})^\ell}  + O\left( q^{3/2+\ell+g}\right). \]
Adding them up, we get a total contribution of
\[ \frac{q^{2g+3 }(1-q^{-1})^{\ell+1}}{(1-q^{-2})^{\ell-1}} + O\left( q^{3/2+ \ell+g}\right).\]

In both cases, the main term is the same, and dividing by
\[ \left| \widehat{ \mathcal F}_{2g+1} \right| + \left| \widehat{\mathcal F}_{2g+2}\right|  = q^{ 2 g + 3} (1-q^{-1}) (1-q^{-2}) \]
we obtain that
\[M_k(q,g) =  \frac{1}{(q+1)^{k/2}} \sum_{l=1}^k c(k,\ell) \sum_{({\mathbf x},{\mathbf  {b}}) \in P_{k,\ell} \atop {b_i \, \mathrm {even} }  }  (1+q^{-1}) ^{-\ell}  + O((q+1)^{-k/2} (q+1)^{k}
q^{-3/2 + k -g}),\]
where the error term is estimated using \eqref{comb}.

On the other hand, the corresponding moment of the normalized sum of our random variables is
\[\mathbb E \left( \left( \frac{1}{\sqrt{q+1}} \sum_{i=1}^{q+1} X_i \right)^k \right) = \frac{1}{(q+1)^{k/2}} \sum_{\ell=1}^k c(k,\ell) \sum_{({\mathbf i},{\mathbf  {b}}) \in A_{k,\ell} }\mathbb E \left( X_{i_1}^{b_1} \cdots X_{i_\ell}^{b_\ell}  \right),
\]
where
\[A_{k,\ell} = \left\{(\mathbf i, \mathbf b); \mathbf i= (i_1, \ldots, i_\ell) , 1\leq i_j\leq q+1,  i_j \textrm{'s distinct }, \mathbf b = (b_1, \ldots, b_\ell), \sum_{j=1}^\ell b_j =k   \right\} \]
is clearly isomorphic to $P_{k,\ell}.$

Since
\[\mathbb E(X_i^b) = \begin{cases}
0 & b \textrm{ odd}\\
\displaystyle \frac{1}{1+q^{-1}} & b \textrm{ even}
\end{cases}\]
and $X_1, \ldots, X_{q+1}$ are independent, we get
\begin{equation}\label{moments}
M_k(q,g) = \mathbb E \left( \left( \frac{1}{\sqrt{q+1}} \sum_{i=1}^{q+1}
X_i \right)^k \right)  + O\left(
q^{(3k-3-2g)/2}\right).\end{equation}

Since the moments of a sum of bounded i.i.d.~random variables
converge to the Gaussian moments \cite[Section 30]{Bi}, it follows that, as $q, g
\rightarrow \infty ,$ $M_k(q,g)$ agrees with  Gaussian moments for
all $k$. Hence the limiting value distribution of
$\Tr(\Frob_C)/\sqrt{q+1}$ is a standard Gaussian distribution with mean 0 and
variance 1.

\section{General case}
\label{gen}

In this section we briefly sketch the proof of our results for  the case of curves $C$ that have a cyclic $p$-to-$1$ map to $\mathbb{P}^1(\mathbb F_q)$, where  $q \equiv 1 \pmod p$ and $p$ is an odd prime.  As we mentioned in the introduction, the proof of the general case follows from the same techniques as in the cyclic trigonal case.

Denote by $\mathcal F_{(d_1, \ldots, d_{r})}$ the set of polynomials of the form $F(X) = F_1(X) F_2^2(X) \cdots F_{r}^{r}(X)$ with $F_1, \ldots, F_{r}$  monic, square-free and pairwise coprime polynomials of degrees $d_1, \ldots, d_{r}$, respectively. We note that when $r=p-1$ this is the set of monic $p$-th power-free polynomials. 
For fixed $x_1, \ldots, x_{\ell}$ distinct points in $\mathbb F_q$ and $a_1, \ldots, a_\ell \in \mathbb F_q^*$,
\begin{equation*}
\left|\left\{F\in \mathcal F_{(d_1, \ldots, d_r)}: F(x_i) = a_i\right\}\right| =
\frac{q^{d_1-\ell}}{\zeta_q(2) (1-q^{-2})^\ell} \sum_{\deg F_2=d_2} \ldots \sum_{\deg F_r=d_r} b(F_2\ldots F_r) + O\left(q^{d_1/2+d_2+\ldots+d_r}\right),
\end{equation*}
where $b(F)$ is the quantity defined in \eqref{b's}. Here we used the fact that $b(F)$ is multiplicative and $b(F_1\dots F_r)=0$ if the $F_i$ are not relatively prime in pairs.

Using the Tauberian theorem and an induction argument on $r$, we obtain the following result which mirrors Proposition~\ref{nonzerovalues}.

\begin{proposition}
Fix $0 \leq \ell \leq q$,  $x_1, \ldots, x_\ell$ distinct points in $\mathbb F_q$ and $a_1, \ldots, a_\ell$ nonzero elements of $\mathbb F_q$. For each $r \geq 2,$
\begin{eqnarray*}
\left|\left\{F\in \mathcal F_{(d_1, \ldots, d_{r})}: F(x_i) = a_i, 1\leq i \leq\ell \right\}\right| &=&\frac{L_{r-1}q^{d_1+\dots + d_r}}{\zeta_q(2)^r } \left(\frac{q}{(q+r)(q-1)} \right)^\ell  \\
&\times&
\left(1 + O\left(q^{\varepsilon(d_2+\dots+d_r+\ell)} \left(q^{-d_{2}}+\dots +q^{-d_r}\right) + q^{-d_1/2+\ell}\right)\right),
\end{eqnarray*}
where
\[L_{r-1} = \prod_{j=1}^{r-1} \prod_P  \left( 1-\frac{j}{(|P|+1)(|P|+j)}\right).\]
\end{proposition}
Denote
\[\mathcal F_{(d_1, \ldots, d_{p-1})}^{(k_1, \ldots, k_{p-1})} = \left\{F=F_1\ldots F_{p-1}^{p-1}\in \mathcal F_{(d_1, \ldots, d_{p-1})}; F_i \textrm { has $k_i$ roots in } \mathbb F_q, 1\leq i \leq p-1 \right \}. \]
Proceeding as in the proof of Corollary \ref{whatever}, we obtain the following.

\begin{corollary} \label{fixedzeros}
Fix $0\leq m \leq q$. Choose $x_1, \ldots, x_{q}$ an enumeration of the points of $\mathbb F_q$, and values $a_{1} = \ldots = a_{m}=0$, $a_{m+1}, \ldots, a_{q} \in \mathbb F_q^*$. Pick a partition $m=k_1+\ldots +k_{p-1}$. Then for any $\varepsilon > 0$,
\begin{eqnarray*}
&& |\{ F \in \mathcal F_{(d_1, \ldots, d_{p-1})}^{(k_1, \ldots, k_{p-1})}: F(x_i)=a_i, 1\leq i\leq q\}| \\
&&=\binom{m}{k_1, \, \ldots \,,  k_{p-1}}  \frac{L_{p-2}q^{d_1+\dots + d_{p-1}}}{\zeta_q(2)^{p-1} }  \left(\frac{1}{q+p-1}\right)^m\left(\frac{q}{(q+p-1)(q-1)}\right)^{q-m}\\
&& \times
\left(1 + O\left(q^{\varepsilon(d_2+\dots+d_{p-1}+k_1-m+q)} \left(q^{-(d_{2}-k_2)}+\dots +q^{-(d_{p-1}-k_{p-1})} \right)+ q^{-(d_1-k_1)/2+q}\right)\right).
 \end{eqnarray*}
\end{corollary}

Summing over all such possible partitions of $m$, just as we did in the proof of Corollary \ref{formoments}, and using the Multinomial Theorem, we obtain that
 \begin{align*}
&\frac{ \left|\left\{F\in \mathcal F_{(d_1, \ldots, d_{p-1})}: F(x_i) = a_i, 1\leq i\leq q\right\}\right|}{ \left| \mathcal F_{(d_1, \ldots, d_{p-1})}\right| }=  \left(\frac{p-1}{q+p-1}\right)^m\left(\frac{q}{(q+p-1)(q-1)}\right)^{q-m} \\
&\times \left(1 + O\left(q^{\varepsilon(d_2+\dots+d_{p-1}+q)+(1-\varepsilon)m} \left(q^{-d_{2}}+\dots +q^{-d_{p-1}}\right) + q^{-(d_1-m)/2+q}\right)\right). \\
\end{align*}

Taking into account the number of elements in each $p$-power residue class in $\mathbb F_q$, we arrive at the following result, which corresponds to Theorem~\ref{countingF} from the cyclic trigonal case.

\begin{theorem}\label{conjp}
Choose $x_1, \ldots, x_q$ an enumeration of the points of $\mathbb F_q$. Fix $\varepsilon_1, \dots, \varepsilon_q \in \mathbb C$, such that $m$ of them are $0$ and the rest are $p$-th roots of unity. Then for any $\varepsilon >0 $,
\begin{align*}
&\frac{ \left|\{F\in \mathcal F_{(d_1, \ldots, d_{p-1})}: \chi_p(F(x_i)) =
\varepsilon_i, 1\leq i\leq q\}\right|}{ \left| \mathcal F_{(d_1, \ldots, d_{p-1})}\right| }=  \left(\frac{p-1}{q+p-1}\right)^m\left(\frac{q}{p(q+p-1)}\right)^{q-m} \\
&\times \left(1 + O\left(q^{\varepsilon(d_2+\dots+d_{p-1}+q)+(1-\varepsilon)m} \left(q^{-d_{2}}+\dots +q^{-d_{p-1}}\right) + q^{-(d_1-m)/2+q}\right)\right). \\
\end{align*}
\end{theorem}

Proceeding as in Section \ref{geometry}, one can prove that the point at infinity behaves like any other point of $\mathbb P^1(\mathbb F_q).$ For a curve $C$ with affine model $Y^p=F(X)$, $F = F_1 F_2^2\ldots F_{p-1}^{p-1}$ the number of branch points is $R = d_1+\ldots+d_{p-1}$ if $\deg F \equiv 0 \pmod p$ or $R= d_1 +\ldots+d_{p-1}+1$ otherwise.  The genus of such a curve is always $g=(p-1)(R-2)/2$. The moduli space breaks into a disjoint union of irreducible components

\[\mathcal H_g = \bigcup \mathcal H^{(d_1, \ldots, d_{p-1})}.\]
Here the components are indexed by tuples $(d_1, \ldots, d_{p-1})$  with the properties that $(p-1)(d_1+\ldots +d_{p-1}-2) =2g$ and $d_1+2d_2+\ldots +(p-1)d_{p-1} \equiv 0 \pmod p$, taking into account the fact that two such tuples give the same component under certain equivalence relations (in the case $p=3$ this amounts to switching $d_1$ and $d_2$). We also remark that, just as in the cyclic trigonal case, for a curve $C$ of genus $g>(p-1)^2$, the cyclic $p$-to-$1$ map to $\mathbb P^1(\mathbb F_q)$ is uniquely determined up to isomorphisms of $\mathbb P^1(\mathbb F_q)$. So when the genus passes this threshold, counting  all possible affine models for curves of a fixed inertia type
(with the appropriate weights) will count each curve with the same multiplicity. Namely,

\[ \left|\mathcal H^{(d_1, \ldots, d_{p-1})}\right|' = \frac{1}{q(q^2-1)} \left( \left|\mathcal F_{(d_1, \ldots, d_{p-1})}\right| +  \left |\mathcal F_{(d_1-1, \ldots, d_{p-1})}\right|+ \ldots +  \left|\mathcal F_{(d_1, \ldots, d_{p-1}-1)}\right|\right).\]

Similar to  the cyclic trigonal curves, the curves $C$ are endowed with an automorphism of order $p$ that splits the first cohomology group $H^1$ into subspaces $H^1_{\chi_p}, H^1_{\chi^2_p}, \ldots, H^1_{\chi^{p-1}_p}$ on which the automorphism acts by multiplication by $\chi_p,\ldots, \chi^{p-1}_p$ respectively (for a choice of an order $p$ character $\chi_p$). Since this automorphism commutes with the action of Frobenius, it suffices to study the trace of Frobenius on one of these subspaces, say $H^1_{\chi_p}.$ Moving to another subspace amounts to a new choice of $\chi_p.$

The trace of Frobenius of the curve $C$ with affine model
\[C: Y^p= F(X)\] on each subspace of $H^1$ is then given by

\[\Tr(\Frob_C|_{H^1_{\chi^j_p}}) = - \sum_{x \in \mathbb P^1(\mathbb F_q)} \chi_p^j(F(x)) = -\widehat S_j(F) \qquad 1\leq j \leq p-1\]
where the value of $F$ at the point at infinity is the value at zero of
$X^{\deg F} F(1/X)$ if $\deg F \equiv 0 \pmod p$ and $0$ otherwise. 
The number of points of $C$ over $\F_q$ (including the points at
infinity) is then
\[q+1 + \left( \widehat S_1(F)+\ldots+\widehat S_{p-1}(F) \right).\]

Following the argument from the proof of Theorem  \ref{componentd1d2}, one can show that the projective trace is distributed just like the affine trace.

\begin{theorem}\label{comp-gen}
Let $X_1, \ldots, X_{q+1}$ be complex i.i.d.~random variables taking the value $0$ with probability $(p-1)/(q+p-1)$ and each of the $p$-th roots of unity in $\mathbb C$ with probability $q/(p(q+p-1)).$
As $d_1, \ldots, d_{p-1} \rightarrow \infty$,
\begin{align*} &\frac{\left\vert \left\{ C \in \mathcal H^{(d_1, \ldots, d_{p-1})} : \Tr(\Frob_C|_{H^1_{\chi_p}}) = -s \right\} \right\vert '}{\left\vert \mathcal H^{(d_1, \ldots, d_{p-1})}\right\vert '}
\\ &= \Prob\left(\sum_{i=1}^{q+1}X_i = s\right)
 \left(1 + O\left(q^{\varepsilon(d_2+\dots+d_{p-1})+q} \left(q^{-d_{2}}+\dots +q^{-d_{p-1}}\right) + q^{-(d_1-3q)/2}\right)\right)
\end{align*}
 for any $s \in  \mathbb C$, $|s|\leq q+1$ and $0 > \varepsilon > 1$.
\end{theorem}

Note that our random variables are complex-valued and have the property that
\begin{equation*}\mathbb E(X_i^b\overline{X_i}^c) = \begin{cases}
0 & b \not \equiv c \, (\mathrm{mod}\, p) ,\\
\displaystyle \frac{q}{q+p-1} & b \equiv c \, (\mathrm{mod}\, p).
\end{cases}\end{equation*}

Computing the mixed moments of the trace, one sees that they approach the moments of the normalized sum of random variables $(X_1+\ldots+X_{q+1})/\sqrt{q+1}.$ Namely, for each $j,k\geq 0,$ denote

\[M_{j,k}(q, (d_1, \ldots, d_{p-1}))= \frac{1}{|\mathcal H^{(d_1, \ldots, d_{p-1})}|'} \sum_{C \in \mathcal H^{(d_1, \ldots, d_{p-1})}} \left( \frac{-\Tr
(\Frob_C|_{H^1_{\chi_p}})}{\sqrt{q+1}} \right)^j \left( \frac{-\Tr
(\Frob_C|_{H^1_{\overline\chi_p}})}{\sqrt{q+1}} \right)^k.  \]

A similar computation to the one in the proof of Theorem \ref{momentsthm} yields

\begin{eqnarray}\label{momgen}
\nonumber M_{j,k}(q, (d_1, \ldots, d_{p-1}))&=&  \mathbb E
\left( \left( \frac{1}{\sqrt{q+1}} \sum_{i=1}^{q+1} X_i \right)^j\left( \frac{1}{\sqrt{q+1}} \sum_{i=1}^{q+1} \overline{X_i} \right)^k
\right)\\
& \times& \left( 1 + O\left(q^{\varepsilon(d_2+\dots+d_{p-1}+j+k)} \left(q^{-d_{2}}+\dots +q^{-d_{p-1}}\right) + q^{-d_1/2+j+k}\right)\right).
\end{eqnarray}

Writing each random variable $X_j$ in terms of its real and imaginary part, $X_j=A_j+\sqrt{-1}B_j,$ we obtain that $\mathbb E(A_j)=\mathbb E (B_j) = 0$ and $\mathbb E(A_j^2)=\mathbb E (B_j^2) = q/(2(q+p-1)).$ Applying the Triangular Central Limit Theorem, we obtain, by the same arguments as in the proof of Corollary \ref{momentscor}, that the limiting distribution of the normalized $(X_1+\ldots+X_{q+1})/\sqrt{q+1}$ is a complex Gaussian with mean $0$ and variance $1$. Together with \eqref{momgen} this fact implies the following result.

\begin{theorem}\label{conjmomp}

As $q, d_1, \ldots, d_{p-1} \rightarrow \infty,$
\[\displaystyle \frac{1}{\sqrt{q+1}} \Tr(\Frob_C|_{H^1_{\chi_p}})\]
has a complex Gaussian distribution with mean $0$ and variance $1$ as $C$ varies in $\mathcal H^{(d_1, \ldots, d_{p-1})}(\mathbb F_q).$
\end{theorem}

\section{Heuristic}
\label{heuristic}

We give in this section a heuristic which explains the probabilities
occurring in
 Theorems \ref{KRrevisitedthm}, \ref{componentd1d2} and \ref{comp-gen}.
\subsection{Heuristic for hyperelliptic curves} \label{heuristichyper}

 We first give a heuristic explaining the results of Lemma \ref{KRproposition6} (Proposition 6 in \cite{KR}). To model square-free polynomials,
 we consider  polynomials with no double root in $\mathbb F_q$. That is, fix points
$x_1, \ldots, x_{\ell+m}$ and count the monic polynomials of
degree $d$ that are not divisible by $(X-x_i)^2$ for any $1 \leq i
\leq \ell+m.$ Assume that $d \gg \ell+m.$ Then by the Chinese Remainder Theorem, the number of such polynomials is the number of monic
polynomials of degree $d$ multiplied by a factor of $(1-q^{-2})$ for
each condition. There are $\ell+m$ conditions, so there are $q^d(1-q^{-2})^{\ell+m}$ polynomials
of degree $d$ which are not divisible by $(X-x_i)^2$ for any $1 \leq i \leq \ell+m$.
We now fix $a_1, \ldots, a_\ell \in \mathbb F^* _q$ and $a_{\ell+1},
\ldots , a_{\ell+m}=0$, and count the number of polynomials defined above
satisfying $F(x_i) = a_i.$ For $i \geq \ell+1$, we want
$F(X) \equiv 0 \bmod (X-x_i)$, and there are $(q-1)$ such residues modulo $(X-x_i)^2$
among the $q^2-1$ residues not congruent to 0 modulo $(X-x_i)^2$.
For $i \leq \ell$, we want
$F(X) \equiv a_i \bmod (X-x_i)$, and there are $q$ such residues modulo $(X-x_i)^2$
among the $q^2-1$ residues not congruent to 0 modulo $(X-x_i)^2$.
Using the Chinese Remainder Theorem, this shows that
\begin{eqnarray}
\nonumber
&&\frac{\left|\left\{F\in\mathbb F_q[X]: \deg F = d, F \textrm{ monic, } (X-x_i)^2 \nmid F, F(x_i) = a_i \right\}\right|}{\left| \left\{F\in\mathbb F_q[X]: \deg F = d, F \textrm{ monic, } (X-x_i)^2 \nmid F \right\}\right|} \\ \nonumber
&&= \left( \frac{q-1}{q^2-1} \right)^m \left( \frac{q}{q^2-1}
\right)^\ell  = \frac{(1-q^{-1})^m q^{-(\ell+m)} } {(1-q^{-2}) ^{\ell+m}},
\end{eqnarray}
which is the main term in Lemma \ref{KRproposition6}.  Then in some way, imposing  the square-free condition cuts
uniformly across these sets, and being square-free is an
event independent of imposing values at a finite number of points.  The error term occurs because if one interprets the square-free condition as a collection of conditions indexed by irreducible polynomials, these individual conditions are only jointly independent in small numbers.

We now illustrate how the above heuristic also explains the probabilities of Theorem \ref{KRrevisitedthm}. This is very similar to the computation of Section \ref{KRrevisited}. As there, we now use the set of (not necessarily monic) polynomials of degree $2g+1$ and $2g+2$. There are $q^{2g+3}(1-q^{-2}) ^{\ell+m+1}$ such polynomials  $F \in \mathbb F_q[X]$ with no double zeros at the points $x_1, \ldots, x_{\ell+m}$. Denoting the point at infinity by $x_{\ell+m+1}$, we have to compute

\begin{eqnarray} \label{computation}
\frac{\left|\left\{F\in\mathbb F_q[X]: \begin{array}{l}2g+1 \leq \deg F \leq 2g+2, (X-x_i)^2 \nmid F, 1\leq i \leq \ell+m \\ F(x_i) = a_i, 1\leq i\leq {\ell+m+1}\end{array} \right\}\right|}{\left| \left\{F\in\mathbb F_q[X]:2g+1 \leq \deg F \leq 2g+2, (X-x_i)^2 \nmid F, 1 \leq i \leq \ell+m \right\}\right|}. \end{eqnarray}
If $F(x_{\ell+m+1})=0$, which is equivalent to $\deg(F)=2g+1$, then the numerator of \eqref{computation} is equal to
$
q^{2g+1-2(\ell+m)}(q-1)q^\ell (q-1)^m.
$
Similarly, if $F(x_{\ell+m+1})\neq 0$, which is equivalent to $\deg(F)=2g+2$, the numerator of \eqref{computation} is equal to
$
q^{2g+2-2(\ell+m)}q^\ell (q-1)^m.
$
This shows that \eqref{computation} is
equal to
\[
 \begin{cases}
 \left( \frac{1}{q+1} \right)^{m+1} \left( \frac{q}{q^2-1} \right)^{\ell} & \text{ if } F(x_{\ell+m+1})=0,
 \\
 \left( \frac{1}{q+1} \right)^{m} \left( \frac{q}{q^2-1} \right)^{\ell+1} & \text{ if } F(x_{\ell+m+1})\neq 0.
 \end{cases}
 \]

This is the geometric version of the main term in Lemma \ref{KRproposition6}. To see that, let $x_1, \dots, x_{q+1}$ be the points of $\mathbb{P}^1(\F_q)$, let $a_1, \dots, a_{q+1} \in \F_q$ and let $m$ be the number of zeros among the values $a_1, \dots, a_{q+1}$. Then \eqref{computation} writes as
\[
\left(\frac{1}{q+1}\right)^m\left(\frac{q}{q^2-1}\right)^{q+1-m}=\frac{(1-q^{-1})^m q^{-(q+1)}}{(1-q^{-2})^{q+1}},
\]
and the probabilities of Theorem \ref{KRrevisitedthm} follow with the usual argument.

\subsection{Heuristic for general case}\label{htri}

The same heuristic can be used to explain the result one gets for curves $C$ that have a cyclic $p$-to-$1$ map to $\mathbb P^1(\mathbb F_q)$,

\begin{lemma}
The number of $(p-1)$-tuples $(F_1, \ldots, F_{p-1})$ of nonzero residues modulo $(X-t)^2$ such that $(X-t)$ does not divide $F_i$ and $F_j$ for any $i\neq j$ is $q^{p-2}(q-1)^{p-1} (q+p-1)$.
\end{lemma}
\begin{proof}  Denote by  $\mathcal S_t$ the set of such tuples. The total number
of  $(p-1)$-tuples of nonzero residues modulo  $(X-t)^2$ is
$(q^2-1)^{p-1}$. For each integer  $1\leq k \leq p-1$, denote
\[\mathcal B_k = \left\{ (F_1, \ldots, F_{p-1}) : (X-t) \textrm{ divides exactly $k$ of the $F_i$, $(X-t)^2$ does not divide any $F_i$}   \right\} .\]
Then \[|\mathcal S_t | = (q^2-1)^{p-1} - \sum_{k=2}^{p-1} |\mathcal
B_k|.  \] It is easy to see that $|\mathcal B_k| = \binom{p-1}{k}
(q-1)^k (q^2 - q)^{p-1-k},$ and the lemma follows by
using the binomial formula.
\end{proof}

The number of $(p-1)$-tuples in $\mathcal S_t$  such that $F=F_1F_2^2\dots F_{p-1}^{p-1}$ takes the value $a\in \F_q^*$ is equal to $q^{p-1}(q-1)^{p-2}$. It follows that the number of $(p-1)$-tuples in $\mathcal S_t$ such that $F=F_1F_2^2\dots F_{p-1}^{p-1}$ takes the value $0$ is equal to $|\mathcal S_t|-q^{p-1}(q-1)^{p-1}=(p-1)q^{p-2}(q-1)^{p-1}$. Therefore the probability that $F=F_1F_2^2\dots F_{p-1}^{p-1}$ takes a
value $a \in \F_q$ at a point $t$ is
\begin{eqnarray} \label{probgeneral}
\left\{ \begin{array}{ll} \displaystyle\frac{p-1}{q+p-1} & \mbox{if $a=0$}, \\
\\
\displaystyle\frac{q}{(q-1)(q+p-1)} & \mbox{if $a \in \F_q^*$}. \end{array} \right.
\end{eqnarray}
This explains the result of Corollary \ref{formoments} and  Theorem \ref{conjp}.

 Finally, for Theorems \ref{componentd1d2} and \ref{comp-gen}, we note that taking the point at infinity into consideration works just like in Section \ref{heuristichyper} and we get that for any enumeration $x_0, \ldots, x_{q}$ of  $\mathbb P^1(\mathbb F_q)$ and any $\varepsilon_0, \ldots, \varepsilon_{q}$ that are either zero or $p$-th roots of unity,
 \[\Prob \left(\chi(F(x_i)) = \varepsilon_i, 0\leq i \leq q \right) = \Prob\left( X_i =  \varepsilon_i, 0\leq i \leq q \right),\]
where $X_0, \ldots, X_{q}$ are i.i.d. random variables taking the value $0$ with probability $(p-1)/(q+p-1)$ and each root of unity with probability $q/(p(q+p-1)).$

{\bf Acknowledgments}. 
This work was initiated at the Banff workshop ``Women in Numbers'' organized  by Kristin Lauter, Rachel Pries, and Renate Scheidler in November 2008, and the authors would like to thank the organizers and BIRS for excellent working conditions. The authors also wish to thank Eduardo Due\~nez, Nicholas Katz, Kiran Kedlaya, P\"{a}r Kurlberg, Lea Popovic, Rachel Pries, and Ze\'{e}v Rudnick  for helpful discussions related to this work.

This work was supported by the Natural Sciences and Engineering Research Council of Canada [B.F., Discovery Grant 155635-2008 to C.D., 355412-2008 to M.L.]; the National Science Foundation of U.S. [DMS-0652529 to A.B.]; and the University of Alberta [Faculty of Science Startup grant to M.L.]

\end{document}